\documentclass[12pt,a4paper,english,reqno]{amsart}
\usepackage{amsthm}
\usepackage{amssymb}
\usepackage{amscd}
\usepackage{amstext}
\usepackage[utf8]{inputenc}

\usepackage{graphicx}
\usepackage{color}

\usepackage[all]{xy}
\xyoption{2cell}
\UseTwocells
\usepackage{yhmath}
\usepackage{tikz}
\usepackage[colorinlistoftodos]{todonotes} 
\theoremstyle{plain}
\newtheorem{thm}{Theorem}[section]
\newtheorem{prp}[thm]{Proposition}
\newtheorem{lem}[thm]{Lemma}
\newtheorem{cor}[thm]{Corollary}

\newtheorem*{thm-nn}{Theorem}
\newtheorem*{prp-nn}{Proposition}
\newtheorem*{lem-nn}{Lemma}
\newtheorem*{cor-nn}{Corollary}
\newtheorem*{clm-nn}{Claim}
\newtheorem*{cnj-nn}{Conjecture}
\newtheorem*{prb-nn}{Problem}

\theoremstyle{definition}
\newtheorem{dfn}[thm]{Definition}
\newtheorem{exm}[thm]{Example}

\newtheorem*{dfn-nn}{Definition}

\newtheorem{rmk}[thm]{Remark}
\newtheorem{ntn}[thm]{Notation}
\newtheorem{asp}[thm]{Assumption}

\newcommand{\xyR}[1]{%
\xydef@\xymatrixrowsep@{#1}}

\newcommand{\xyC}[1]{%
\xydef@\xymatrixcolsep@{#1}}

\def\al{\alpha}
\def\be{\beta}

\def\la{\lambda}

\def\si{\sigma}

\def\top{\operatorname{top}} 
\def\soc{\operatorname{soc}}

\def\Im{\operatorname{Im}}

\def\Hom{\operatorname{Hom}}
\def\rad{\operatorname{rad}}

\def\End{\operatorname{End}}

\def\mod{\operatorname{mod}}

\def\thick{\operatorname{thick}}
\def\add{\operatorname{add}}

\def\Camb{\operatorname{\textbf{Camb}}}
\def\biCamb{\operatorname{\textbf{biCamb}}}

\def\Kb{{\mathcal K}^{\text{\rm b}}}

\def\calP{{\mathcal P}}
\def\calQ{{\mathcal Q}}

\def\calS{{\mathcal S}}

\def\bbA{{\mathbb A}}

\def\bbZ{{\mathbb Z}}

\def\bbR{{\mathbb R}}

\def\incl{\hookrightarrow}
\def\iso{\cong}
\def\ds{\oplus}

\def\Ds{\bigoplus}

\def\dsm#1,#2..#3{\bigoplus_{{#1}={#2}}^{#3}}
\def\sm#1,#2..#3{\sum_{{#1}={#2}}^{#3}}

\def\id{1\kern-.25em{\text{{\rm l}}}} 

\def\isoto{\ \raise.3ex\hbox{$^{\sim}$}\kern-.8em\hbox{$\to$}\ } 
\def\ya#1{\xrightarrow{#1}}

\def\bg{%
\family{cmr}\size{20}{12pt}\selectfont}

\def\bigzerou{%
\smash{\lower1.7ex\hbox{\bg 0}}}

\def\repr[#1;#2;#3;#4;#5]{
\left(
\begin{matrix}#1\\#2\end{matrix}
#3
\begin{matrix}#4\\#5\end{matrix}
\right)}
\def\pmat#1{\begin{pmatrix} #1 \end{pmatrix}}

\def\smat#1{\begin{smallmatrix} #1 \end{smallmatrix}}
\def\psmat#1{\left(\begin{smallmatrix} #1 \end{smallmatrix}\right)}

\def\k{\Bbbk}

\def\tsilt{\mbox{\rm 2-silt}\hspace{.01in}}
\def\ttilt{\mbox{\rm 2-tilt}\hspace{.01in}}
\def\tptilt{\mbox{\rm 2-ptilt}\hspace{.01in}}
\def\tpjtilt{\mbox{\rm 2-ptilt}^j\hspace{.01in}}
\def\indtptilt{\mbox{\rm 2-iptilt}\hspace{.01in}}

\def\proj{\operatorname{proj}}
\def\conv{\operatorname{conv}}

\keywords{Brauer tree algebras, 2-term tilting complexes, simplicial complexes, derived invariants}

\title{Simplicial complexes and tilting theory for Brauer tree algebras}
\author{Hideto Asashiba}
\address{Department of Mathematics, Faculty of Science, Shizuoka University,
  836 Ohya, Suruga-ku, Shizuoka, 422-8529, Japan}
\email{asashiba.hideto@shizuoka.ac.jp}
\thanks{H.A.\ is supported by Grant-in-Aid for Scientific Research (C) 18K03207}

\author{Yuya Mizuno}
\address{Faculty of Liberal Arts and Sciences, Osaka Prefecture University, 1-1 Gakuen-cho,
Naka-ku, Sakai, Osaka 599-8531, Japan}
\email{yuya.mizuno@las.osakafu-u.ac.jp}
\thanks{Y.M.\ is supported by Grant-in-Aid for JSPS Research Fellow 17J00652.}

\author{Ken Nakashima}
\address{Department of Mathematics, Faculty of Science, Shizuoka University, 
836 Ohya, Suruga-ku, Shizuoka, 422-8529, Japan}
\email{nakashima.ken@shizuoka.ac.jp}

\begin{document}
\maketitle

\begin{abstract}
We study 2-term tilting complexes of Brauer tree algebras in terms of simplicial complexes. 
We show the symmetry and convexity of the lattice polytope corresponding to the simplicial complex of 2-term tilting complexes. 
Via a geometric interpretation of derived equivalences, 
we show that the $f$-vector of the simplicial complexes of Brauer tree algebras only depends on the number of the edges of the Brauer trees and hence it is a derived invariant. 
In particular, this result implies that the number of 2-term tilting complexes, which is in bijection with support $\tau$-tilting modules, is a derived invariant.
Moreover, we apply our result to the enumeration problem  of Coxeter-biCatalan combinatorics.
\end{abstract}

\section{Introduction}

In this paper, we study derived invariants for Brauer tree algebras via simplicial complexes defined by a family of tilting complexes. Let us briefly recall the historical background on Brauer tree algebras (Definition \ref{Brauer tree}). 
Brauer tree algebras can trace its history to modular representation theory. Brou\'{e}'s abelian defect group conjecture asserts that a block $B$ of a (nonsemisimple) finite group algebra with an abelian defect $D$ would be derived equivalent to its Brauer correspondent $B'$. 
If $D$ is cyclic, then both blocks $B$ and $B'$ turns out to be algebras given by Brauer trees having the same number of edges and multiplicities (we refer to \cite[Section V]{Al}). 
From this viewpoint, Rickard showed that 
Brauer tree algebras are derived equivalent if and only if they 
have the same number of edges and multiplicities,  
and gave an affirmative answer to this conjecture for the cyclic defect case \cite{R2}. 
 
In our work, we study the set of \emph{2-term tilting complexes} (Definition \ref{two-term tilt}) for Brauer tree algebras and give a better understanding for their derived invariants. 
The poset of 2-term tilting complexes are particularly important in $\tau$-tilting theory, introduced by Adachi-Iyama-Reiten \cite{AIR}. 
Indeed, in the case of symmetric algebras, this poset is
isomorphic to that of support $\tau$-tilting modules, and their mutation behaviors are much better than classical tilting theory (see expository papers \cite{IR,BY} about $\tau$-tilting theory, cluster theory and many related topics). 
Moreover, Brauer tree algebras are symmetric and representation-finite.   It implies that the set of tilting complexes are transitive by the action of mutation  (more strongly, it is \emph{tilting-discrete}) and the mutation behavior of 2-term tilting complexes are essential for the whole mutation behavior of tilting complexes \cite{AM}. 

Our key method for this study is a realization of the set of 2-term tilting complexes as a simplicial complex and a lattice polytope. 
We investigate the simplicial complex via the geometric realization. To explain our results, we give the following set-up.

Let $G$ be a Brauer tree having $n$ edges with an arbitrary multiplicity, $A_G$ the Brauer tree algebra and $\tptilt(A_G)$ 
the set of isoclasses of basic 2-term pretilting complexes of $\Kb(\proj A_G)$ (see Definition \ref{two-term tilt} for details).
For an integer $j$ such that $1\leq j \leq n$, we let
$$
\tpjtilt (A_G):=\{T\in \tptilt (A_G) \mid |T|=j \},
$$
where $|T|$ denotes the number $d$ in the direct sum decomposition $T = \Ds_{i=1}^d T_i$ of $T$ into indecomposable direct summands $T_i$.
Note that $\tptilt^n(A_G)$ coincide with the set of 2-term tilting complexes \cite[Proposition 3.3]{AIR}. 
Then, following \cite{DIJ}, we define the simplicial complex $\Delta=\Delta(A_G)$ on the following set : 
$$\Delta^0:=\{[T]\ |\ T\in\tptilt^1(A_G)\},$$
where $[T]$ is an element of the Grothendieck group of $K_0(\Kb(\proj A_G))$. 
Then, for a subset $\{[T_1],\ldots,[T_j]\}$ of $\Delta^0$, we declare the set to be a simplex of $\Delta$ if 
$T_1\oplus\cdots\oplus T_j\in\tptilt (A_G)$.
Hence the set of $j$-dimensional faces $\Delta^{j}$ correspond to $\tptilt^{j+1}(A_G)$, 
and the $f$-vector $(f_0,f_1,\cdots,f_{n-1})$, which is defined by $f_{j}=\#\Delta^{j}$, presents the explicit number of 2-term pretilting complexes. 

On the other hand, we can define the \emph{$g$-vector} of $T$ (Definition \ref{g-vector})
by $g(T):= (g_1, \cdots, g_n)^t \in \bbZ^n$
if $[T]=\sum_{i=1}^n g_i[e_iA]$ in $K_0(\Kb(\proj A_G))$, the convex hull of $T$ by 
$$\conv_0(T):= \conv(0, g(T_1), \dots, g(T_n)),$$ 
and the {\em $g$-polytope} of $A_G$ by 
$$\mathcal{P}(A_G):=\bigcup_{T \in \tiny\ttilt A_G}\conv_0(T).$$ 
{Then $\mathcal{P}(A_G)$ is a lattice polytope of $\bbR^n$ admitting a 
unimodular triangulation (Proposition \ref{volume}) and 
$\Delta(A_G)$ is the simplicial complex determined by the unimodular triangulation of $\tptilt^n(A_G)$.}

Our first crucial observation is in the symmetry of the polytope relative to the origin (Corollary \ref{same convex}). 
This fact gives a correspondence of the upper half part and the lower half part of the $f$-vectors of $\Delta(A_G)$ divided by $H_i^0:=\{(v_j)_j^n \in \bbR^n \mid v_i = 0\}$ (Theorem \ref{thm:main1}). 
The second key observation is in the relationship between  
the shapes of $g$-polytopes for derived equivalent algebras.
More precisely, for two derived equivalent algebras $A_G$ and $A_{\mu_i(G)}$, where $\mu_i(-)$ denotes the Kauer move (see \cite{K}), 
we show that the derived equivalence functor induces a correspondence between the upper half part of $\Delta(A_G)$  and the lower half part of $\Delta(A_{\mu_i(G)})$ (Lemma \ref{Lem:num}).  
From these two observations, 
we show that the $f$-vector (and even for any $f_j$) is a derived invariant for Brauer tree algebras and independent of the shape of $G$.

The following picture describes $\calP(A_G)$ (and $\Delta(A_G)$) in the case of a linear  tree $G$ having 3 edges  
(see Example \ref{exam1}).

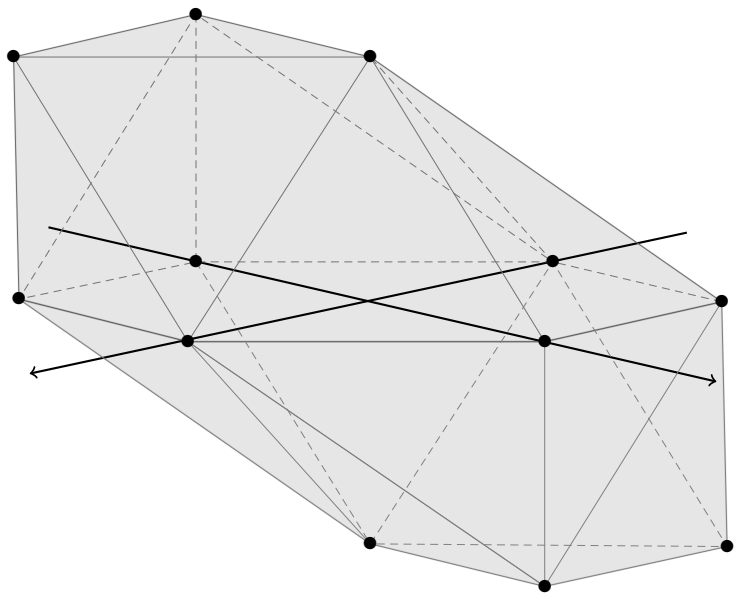
\[
\begin{tikzpicture}
\draw[draw=black,->,thick] (118bp,26bp)--(-127bp,-27bp);
\draw[draw=black,->,thick] (-120bp,28bp)--(129bp,-30bp);
\draw[draw=gray,ultra thin,densely dashed] (131bp,0bp)--(68bp,15bp)--(-65bp,15bp)--(-131bp,1bp)--(-65bp,108bp)--(68bp,15bp)--(0bp,92bp);
\draw[draw=gray,ultra thin,densely dashed] (-65bp,108bp)--(-65bp,15bp);
\draw[draw=gray,ultra thin] (0bp,92bp)--(-65bp,108bp)--(-133bp,92bp)--(-131bp,1bp)--(-68bp,-15bp)--(65bp,-15bp)--(131bp,0bp)--(0bp,92bp)--(-133bp,92bp)--(-68bp,-15bp)--(0bp,92bp)--(65bp,-15bp);
\filldraw[draw=black,thick,fill=black,opacity=.1] (-131bp,1bp)--(-68bp,-15bp)--(65bp,-15bp)--(131bp,0bp)--(0bp,92bp)--(-65bp,108bp)--(-133bp,92bp)--cycle;
\filldraw[draw=black,thick,fill=black,opacity=.1] (-131bp,1bp)--(-68bp,-15bp)--(65bp,-15bp)--(131bp,0bp)--(133bp,-92bp)--(65bp,-107bp)--(0bp,-91bp)--cycle;
\draw[draw=gray,ultra thin] (-131bp,1bp)--(0bp,-91bp)--(-68bp,-15bp)--(65bp,-107bp)--(131bp,0bp)--(133bp,-92bp)--(65bp,-107bp)--(0bp,-91bp);
\draw[draw=gray,ultra thin] (-68bp,-15bp)--(65bp,-107bp)--(65bp,-15bp);
\draw[draw=gray,ultra thin,densely dashed] (-65bp,15bp)--(0bp,-91bp)--(68bp,15bp)--(133bp,-92bp)--(0bp,-91bp);
\draw (-68bp,-15bp) node (P1) {$\bullet$};
\draw (65bp,-15bp) node (P2) {$\bullet$};
\draw (0bp,92bp) node (P3) {$\bullet$};
\draw (68bp,15bp) node (Q1) {$\bullet$};
\draw (-65bp,15bp) node (Q2) {$\bullet$};
\draw (-131bp,1bp) node (S1) {$\bullet$};
\draw (131bp,0bp) node (T1) {$\bullet$};
\draw (-133bp,92bp) node (T2) {$\bullet$};
\draw (-65bp,108bp) node (S3) {$\bullet$};
\draw (0bp,-91bp) node (Q3) {$\bullet$};
\draw (65bp,-107bp) node (T3) {$\bullet$};
\draw (133bp,-92bp) node (S2) {$\bullet$};
\end{tikzpicture}
\]

 Our main results are summarized as follows.

\begin{thm}[Theorem \ref{main3}]
Let $G$ be a Brauer tree having $n$ edges with an arbitrary multiplicity and 
$A_G$ the Brauer tree algebra of $G$. 
Then the $f$-polynomial and $h$-polynomial {\rm (in Definition \ref{f-polynomial})} of $\Delta(A_G)$ are given as follows$:$
$$f(x) = \sum_{j=0}^{n}\binom{n+j}{j,j,n-j}x^{n-j},\ \ \ h(x) = \sum_{j=0}^{n}\binom{n}{j}^2x^{n-j},$$
where we denote by $\binom{n+j}{j,j,n-j}:=(n+j)!/j! j! (n-j)!$.
In particular, the number of 2-term tilting complexes is $\binom{2n}{n}$ and it is a derived invariant.
\end{thm}


We remark that if a Brauer tree is star-shaped (resp.\ line), then it is a Nakayama algebra (resp.\ zigzag algebra). In these cases, it is shown that the number of 2-term tilting complexes is equal to $\binom{2n}{n}$ by Adachi \cite{Ad} (resp.\ by Aoki \cite{Ao1}). 
Thus, our result can be regarded as a uniform treatment for any Brauer tree algebra. 

As a consequence of our result, we apply the result to the enumeration problem of Coxeter-biCatalan combinatorics studied by Barnard-Reading \cite{BR}. 
This is a ``twin version'' of Coxeter-Catalan combinatorics such as noncrossing partitions, clusters, Cambrian lattices and sortable elements (we refer to the original paper \cite{BR} and also \cite{Re1,Re2,FR} for the interesting background of Coxeter-Catalan combinatorics). 
Let $W$ be the Weyl group of type $\bbA_{n}$ and $c$ a bipartite Coxeter element of $W$. 
In \cite{BR}, it is shown that the \emph{biCambrian fan} $\biCamb(W,c)$ is a simplicial fan and the $h$-vector of the simplicial sphere underlying $\biCamb(W,c)$ is determined \cite[Theorem 2.13]{BR}. 
Applying our result above, we give an alternative proof as follows.

\begin{cor}[Theorem \ref{bicatalan}]
Let $W$ be the Weyl group of type $\bbA_{n}$ and 
$c$ a bipartite Coxeter element of $W$. 
The $f$-polynomial and $h$-polynomial of 
the simplicial sphere underlying the biCambrian fan $\biCamb(W,c)$ are given by the following formulas$:$
$$
f(x) = \sum_{j=0}^{n}\binom{n+j}{j,j,n-j}x^{n-j},\ \ \ h(x) = \sum_{j=0}^{n}\binom{n}{j}^2x^{n-j}.
$$
\end{cor}

While working on this project, we were informed that Toshitaka Aoki also determines the number of 2-term tilting complexes of Brauer tree algebras by an entirely different method \cite{Ao2}. We thank him for sharing his knowledge. 

\section{Preliminaries}
In this section, we recall some terminologies related to simplicial complexes and 
Brauer tree algebras, 2-term tilting complexes.


\subsection{Brauer tree algebras}

We first summarize some definitions related to graphs.

\begin{dfn}
A {\em graph} is a triple $G:= (G_0, G_1, C_G)$ of sets $G_0, G_1$
and a map $C_G \colon G_1 \to \{\{x, y\} \mid (x, y) \in G_0 \times G_0\}$.
We usually draw the graph $G$ as a picture with
vertices $x$ bijectively corresponding to elements in $G_0$
and with edges $a$ bijectively corresponding to elements in $G_1$
that connects vertices $x$ and $y$
if and only if $C_G(a) = \{x, y\}$.
If $x = y$, then the edge $a$ connects $x$ and $x$, which is called a {\em loop}.
\begin{enumerate}
\item
Elements of $G_0$ (resp.\ $G_1$) are called {\em vertices} (resp.\ {\em edges}) of $G$.
\item
For each $a \in G_1$ and $x \in G_0$ we say that $a$ is {\em connected} to $x$ (or $x$ is connected to $a$) if $x \in C_G(a)$.

\item
A {\em subgraph} of $G$ is a graph $H = (H_0, H_1, C_H)$ such that $H_0 \subseteq G_0, H_1 \subseteq G_1$ and $C_H$ is the restriction of $C_G$ to $H_1$.
\end{enumerate}
\end{dfn}

\begin{exm}
Let $n$ be a non-negative integer.
Then $\bbA_{n+1}$ is the graph $G = (G_0, G_1, C_G)$, where
$G_0:= \{0, 1, \dots ,n\}$, $G_1:=\{\al_1, \al_2, \dots, \al_{n}\}$
and $C_G(\al_i):= \{i-1, i\}\ (1 \le i \le n)$.
Then $\bbA_{n+1}$ is presented by the following picture.
$$
\xymatrix{
0 \ar@{-}[r]^{\al_1} & 1 \ar@{-}[r]^{\al_2} &\cdots \ar@{-}[r]^{\al_{n}} & n.
}
$$
\end{exm}

\begin{dfn}
Let $G, G'$ be graphs and $n$ a non-negative integer.
\begin{enumerate}
\item
A graph morphism from $G$ to $G'$ is a pair
$f = (f_0, f_1)$ of maps $f_0\colon G_0 \to G'_0$ and
$f_1\colon G_1 \to G'_1$ such that $f_0(C_G(a)) = C_{G'}(f_1(a))$
for all $a \in G_1$.
By abuse of notation we write both $f_0$ and $f_1$ simply by $f$.

\item
A graph morphism $W = (W_0, W_1) \colon \bbA_{n+1} \to G$ is called an {\em oriented walk} in
$G$ of {\em length} $n$ from $W(0)$ to $W(n)$.
An oriented walk $W$ is called a {\em cycle} if $W(0) = W(n)$ and $n \ge 1$, and $W$ is said to be {\em simple} if $W_0$ is injective (note in this case that 
automatically $W_1$ turns out to be also injective).

\item
$G$ is said to be {\em connected} if for each pair $(x, y)$ of
vertices of $G$, there exists an oriented walk from $x$ to $y$.

\item
$G$ is called a {\em tree} if there exists no cycles in $G$.

\item
The {\em image} $w$ of a simple oriented  walk $W$ is the unique subgraph $w = (w_0, w_1, C_w)$ of $G$ such that
$w_0 = \Im W_0$ and $w_1 = \Im W_1$, which is called a {\em walk} of length $n$ between $W(0)$ and $W(n)$.

\item
If we put $a_i:= W(\al_i)$ for all $i$, then a walk $w$ of length $n$ ($\ge 1$) given by an oriented walk $W$ 
turns out to be a subgraph of $G$ of the form
$$
\xymatrix{
x_0 \ar@{-}[r]^{a_1}& x_1 \ar@{-}[r]^{a_2} &\cdots \ar@{-}[r]^{a_n} &x_n,
}
$$
which is denoted by
$$
a_1a_2\cdots a_n
.$$
\end{enumerate}
\end{dfn}

In this paper {\em we always assume that oriented walks are simple}, and hence we may regard walks as subgraphs.

\begin{dfn}
Let $G$ be a tree and $H, H'$ be subgraphs of $G$.
Then $H \cap H'$ (resp.\ $H\cup H'$) is the unique subgraph of $G$
with vertex set $H_0 \cap H'_0$ (resp.\ $H_0 \cup H'_0$)
and with edge set $H_1 \cap H'_1$ (resp.\ $H_1 \cup H'_1$).
These notations apply to walks in $G$ as subgraphs of $G$ to have
definitions of $w \cap w'$ and $w \cup w'$ for walks $w, w'$ in $G$.
Note that if $w_0 \cap w'_0 \ne \emptyset$.
then $w \cap w'$ is again a walk in $G$.
\end{dfn}

Next we recall the definition of Brauer tree algebras. 
We refer to the survey paper \cite{S2} for the background of Brauer tree (graph) algebras. 
In particular, we remark that Brauer tree algebras are representation-finite symmetric algebras (see \cite{S1} for example).

Recall that a {\it cyclic ordering} on a finite set $V$ is a cyclic permutation
of $V$ of order $\#V$.
Note that if $\#V = 1$, then it is just the identity of $V$.

\begin{dfn}\label{Brauer tree} 
We fix a field $\k$.
\begin{enumerate}
    \item A {\em Brauer tree} is a triple $(G, m, S)$, where 
      \begin{enumerate}
          \item $G = (G, \si)$ is a pair consisting of a finite connected tree $G$
          and a family $\si = (\si_x)_{x \in G_0}$ of cyclic 
          orderings $\si_x$ on the set $U_x$ of edges connected to each vertex $x$;
          \item $m$ is an integer $\ge 1$, called the {\it multiplicity}; and
          \item $S$ is a vertex of $G$, called the {\it exceptional vertex}.
      \end{enumerate}
       We usually present the pair $(G, \si)$ by an embedding of $G$ into a plane that gives $\si_x$ as the counterclockwise
    cyclic ordering on the set of edges connected to $x$, and
    usually denote this pair just by $G$.
    We set $|G|:= \# G_1$.
    \item A {\em Brauer quiver} $Q_G = (Q_0, Q_1, s, t)$ associated to a Brauer tree $(G,m,S)$
    is defined as follows:
   {  \begin{enumerate}
        \item $Q_0:= G_1$,
        \item $Q_1:= \{(\si_x(i),i) \mid i\in G_1, x \in C_G(i)\}$,
        \item $s((\si_x(i),i)):= i, t((\si_x(i),i)):= \si_x(i)$ for all $i\in G_1, x \in C_G(i)$.
    \end{enumerate}
        Namely we have arrows $(\si_x(i),i)\colon i \to \si_x(i)$ for all $i\in G_1, x \in C_G(i)$.
    Note that there is a bijection from the set $G_0$
    to the set of oriented cycles in $Q_G$
    sending each vertex $x \in G_0$ to the oriented cycle
    $$
    i \to \si_x(i) \to \si_x^2(i) \to \dots \to \si_x^{n_x}(i) = i,
    $$
    where $i \in U_x$ and $n_x:= \# U_x$.
   In particular if $x$ is connected to only one edge $i$, then there is one loop $i \to i$ because $\si_x(i) = i$.}
Since $G$ is a tree, we have a coloring of $G_0$ by two colors,
say $\al$ and $\be$ such that if an edge is connected to vertices $x$ and $y$, 
then the color of these two vertices are different.
The coloring is unique up to the exchange of $\al$ and $\be$.
Oriented cycles corresponding to vertices with color $\al$ (resp.\ $\be$)
are called $\al$-cycles (resp.\ $\be$-cycles).
We assume that the oriented cycle corresponding to the exceptional vertex $S$ is an $\al$-cycle.
    %
    
    \item Now let $i$ be a vertex of $Q_G$. 
    Let $C_{\alpha}(i)$ be the $\al$-cycle starting at $i$ and $C_{\beta}(i)$ the $\be$-cycle starting at $i$. 
    Moreover, let $\al_i$ be the arrow belonging to $C_{\alpha}(i)$ starting at $i$, 
    $\al^i$ the arrow belonging to $C_{\alpha}(i)$ ending at $i$,
    $\be_i$ the arrow belonging to $C_{\beta}(i)$ starting at $i$ 
    and $\be^i$ the arrow belonging to $C_{\beta}(i)$ ending at $i$.
    Then the {\em Brauer tree algebra} associated to a Brauer tree $G=(G,m,S)$ is defined by the algebra $A_G=\Bbbk Q_{G} /I_G$, 
    where the ideal $I_G$ of $\Bbbk Q_G$ is generated by the elements:
    $\al_i \be^i$, $\be_i \al^i$, 
    $C_{\alpha}(S)^m - C_{\beta}(S)$ and $C_{\alpha}(i) - C_{\beta}(i)$ if $i\not=S$.
\end{enumerate}
\end{dfn}

\begin{rmk}
When the multiplicity ($> 1$) is not on a vertex of degree 1, then we can delete all loops from the quiver using the commutativity relations,
and in this way we obtain a quiver $Q$ without loops and an admissible ideal $I$ of $\k Q$
such that $A_G\iso \k Q/I$
(see the example below).
\end{rmk}

\begin{exm}
Consider the following Brauer tree $G$ with the multiplicity 1:
$$
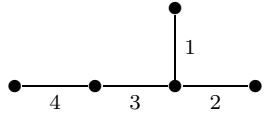
\vcenter{
\xymatrix@M=0pt{
&& \bullet\\
\bullet&\bullet& \bullet &\bullet
\ar@{-}"1,3";"2,3"^1
\ar@{-}"2,3";"2,4"_2
\ar@{-}"2,1";"2,2"_4
\ar@{-}"2,2";"2,3"_3
}}
$$ 
First we draw arrows of the Brauer quiver between edges of $G$
(= vertices of $Q_G$):
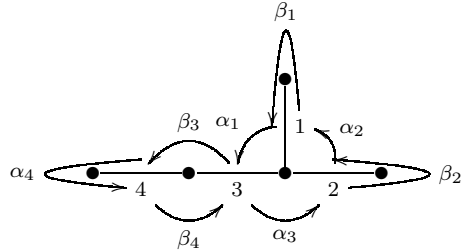
$$
\vcenter{
\xymatrix@M=0pt@C=30pt@R=30pt{
&& \bullet\\
\bullet&\bullet& \bullet &\bullet
\ar@{-}"1,3";"2,3"^{1}="1"_{}="1'"
\ar@{-}"2,3";"2,4"_2="2"_{}="2'"
\ar@{-}"2,1";"2,2"_4="4"^{}="4'"|{}="4''"
\ar@{-}"2,2";"2,3"_3="3"^{}="3'"
\ar@/_/"2'";"1"_{\al_2}
\ar@/_/"1'";"3'"_{\al_1}
\ar@/_1pc/"3";"2"_{\al_3}
\ar@/_1pc/"3'";"4'"_{\be_3}
\ar@/_1pc/"4";"3"_{\be_4}
\ar@/_3pc/"2";"2'"+<0pt,5pt>_{\be_2}
\ar@/_3pc/"4'"+<0pt,5pt>;"4"_{\al_4}
\ar@/_3pc/"1";"1'"+<-5pt,0pt>_{\be_1}
}}
$$ 
Second we change the edges to vertices to obtain the Brauer quiver $Q_G$:
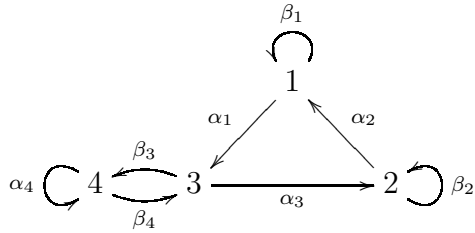
$$
\xymatrix{
&&1 \ar@(ru,lu)[]_{\be_1}\\
4 \ar@(ul,dl)[]_{\al_4}& 3 && 2 \ar@(dr,ur)[]_{\be_2}
\ar"1,3";"2,2"_{\al_1}
\ar"2,2";"2,4"_{\al_3}
\ar"2,4";"1,3"_{\al_2}
\ar@/_/"2,2";"2,1"_{\be_3}
\ar@/_/"2,1";"2,2"_{\be_4}
}
$$
The ideal $I_G$ is given by the following generators:
8 zero relations
$$
\al_1\be_1,\ \be_1\al_2, \ \al_2\be_2,\  \be_3\al_1,\ 
\al_3\be_4,\ \be_2\al_3, \ \al_4\be_3,\ \be_4\al_4
$$
and 4 commutativity relations
$$
\al_2\al_3\al_1 - \be_1, \al_3\al_1\al_2 - \be_2,
\al_1\al_2\al_3 - \be_4\be_3, \al_4 - \be_3\be_4.
$$
As explained in the remark above we can delete loops to make $I_G$
an admissible ideal by using the commutativity relations involving loops.
The final quiver is given as
$$
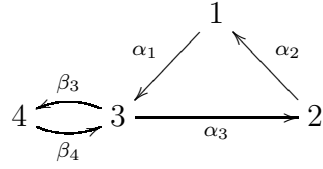
\xymatrix{
&&1 \\
4 & 3 && 2
\ar"1,3";"2,2"_{\al_1}
\ar"2,2";"2,4"_{\al_3}
\ar"2,4";"1,3"_{\al_2}
\ar@/_/"2,2";"2,1"_{\be_3}
\ar@/_/"2,1";"2,2"_{\be_4}
}
$$
with relations
$$
\al_1\al_2\al_3\al_1,\ \al_2\al_3\al_1\al_2, \ \al_2\al_3\al_1\al_2,\  \be_3\al_1,\ 
\al_3\be_4,\ \al_3\al_1\al_2\al_3, \ \be_3\be_4\be_3,\ \be_4\be_3\be_4,\ 
\al_1\al_2\al_3 - \be_4\be_3.
$$
The structure of the projective indecomposable right module $P_i$ corresponding to a vertex $i$ is given as follows (note that right modules are presented by representations of the opposite quiver):
$$
P_1 =\psmat{1\\2\\3\\1},
P_2 = \psmat{2\\3\\1\\2},
P_3 =\left(\!\!\!\!{\tiny
\vcenter{\xymatrix@C=5pt@R=0pt{
&3\\
1\\
&&4\\
2\\
&3
\ar@{-}"1,2";"3,3"
\ar@{-}"3,3";"5,2"
\ar@{-}"1,2";"2,1"
\ar@{-}"2,1";"4,1"
\ar@{-}"4,1";"5,2"
}}}\!\!\!\!\right),
P_4 = \psmat{4\\3\\4},
$$
where each $i \in \{1,\dots, 4\}$ in the diagram stands for a basis vector in the vector space corresponding to the vertex $i$ and mapped by arrows downward (along the lines if they are drawn between basis vectors as in $P_3$).
Since $A_G = P_1 \ds \cdots \ds P_4$ as a right $A_G$-module, we sometimes say that the (right) structure of $A_G$ is given by
$$
\psmat{1\\2\\3\\1}\ds
\psmat{2\\3\\1\\2}\ds
\left(\!\!\!\!{\tiny
\vcenter{\xymatrix@C=5pt@R=0pt{
&3\\
1\\
&&4\\
2\\
&3
\ar@{-}"1,2";"3,3"
\ar@{-}"3,3";"5,2"
\ar@{-}"1,2";"2,1"
\ar@{-}"2,1";"4,1"
\ar@{-}"4,1";"5,2"
}}}\!\!\!\!\right)\ds
\psmat{4\\3\\4}.
$$
\end{exm}



\begin{asp}
Throughout the rest of this paper, unless otherwise stated, 
we only consider Brauer trees with multiplicity 1, which 
does not lose any generality (see Proof of Theorem \ref{proof4.1} for details).
\end{asp}

\begin{ntn}\label{notation}
Throughout the rest of this paper, 
{let $G$ be a Brauer tree and $A_G$ the associated Brauer tree algebra with the Brauer quiver $Q_{G}=:Q$.
We set $n:= |G| = \# Q_0$, and
$Q_0 = \{1,\dots, n\}$.
For each $i \in Q_0$ the idempotent of $A_G$ corresponding to $i$ is denoted by $e_i$.
If there is no confusion, we denote $A_G$ just by $A$.}

{Now let $A$ be an algebra.
We set $|A|$ to be the number $d$ in the direct sum decomposition $A = \Ds_{i=1}^d P_i$ of $A$ into indecomposable direct summands $P_i$.}
We denote by $\mod A$ the abelian category of finite-dimensional  right 
$A$-modules 
and by $\proj A$ the full subcategory of $\mod A$ consisting of projective modules.
Moreover, we denote by $\Kb(\proj A)$ the bounded homotopy category of $\proj A$. 
\end{ntn}


\subsection{2-term tilting complexes}
We recall basic definitions of 2-term tiling complexes.

\begin{dfn}\label{two-term tilt}
\begin{enumerate}
\item
We call a complex $P = (P^i,d^i)$ in $\Kb(\proj A)$ \emph{2-term} if 
$P^i = 0$ for all $i\not= 0,-1$. 
\item We call a complex $P$ in $\Kb(\proj A)$ \emph{pretilting} if $\Hom_{\Kb(\proj A)}(P,P[i])=0$ for any $i\neq0$.

\item We call a complex $P$ in $\Kb(\proj A)$ \emph{tilting} if it is pretilting and satisfies the condition that $\thick(P)$ = $\Kb(\proj A)$, where $\thick(P)$ is the smallest full subcategory of $\Kb(\proj A)$ which contains $P$ and is closed under cones, shifts, direct summands and isomorphisms.
\end{enumerate}
\end{dfn}

We denote by $\tptilt (A)$ (resp.\ $\indtptilt (A)$, $\ttilt (A)$)
the set of isoclasses of basic 2-term pretilting complexes (resp.\ indecomposable 2-term pretilting complexes, 2-term tilting complexes) of $\Kb(\proj A)$. 

\begin{rmk}
Many arguments in section 2 also work for 2-term silting complexes 
of $\Kb(\proj A)$ for any finite dimensional algebra $A$.
In this paper, we only deal with symmetric algebras and therefore silting complexes coincide with tilting complexes. 
To avoid the confusion, we only discuss the case of 2-term tilting complexes. 
\end{rmk}

Let $j$ be an integer such that $1\leq j\leq n$.
We define 
$$\tpjtilt (A):=\{T\in \tptilt (A) \mid |T|=j \}.$$ 
In particular, we have $\tptilt^1 (A) = \indtptilt (A)$ and  
$\tptilt^n (A) = \ttilt (A)$ (see \cite[Proposition 3.3]{AIR}). 
We recall the following basic property.

\begin{lem}\cite[Proposition 2.15]{AIR}\label{no common}
For any $T\in \tptilt(A)$, we write $T= (\cdots 0 \to T^{-1} \to T^{0} \to 0 \cdots )$. 
Then we have $\add (T^0) \cap \add (T^{-1}) = {0}$. 
\end{lem}

Then the following classes of 2-term pretilting complexes play important role\textcolor{red}{s} in this paper.

\begin{dfn}\label{def upper tilt}
Let $i \in Q_0$.
For $T\in \tptilt(A)$, we write $T= (\cdots 0 \to T^{-1} \to T^{0} \to 0 \cdots )$. 
Then we denote by 
$$
\begin{aligned}
\tptilt(A)_i^\leq &:= \{ T\in \tptilt(A) \mid e_i A \not\in \add T^{0} \},\\
\tptilt(A)_i^\geq &:= \{ T\in \tptilt(A) \mid e_i A \notin \add T^{1} \}, \\
\tptilt(A)_i^0 &:=(\tptilt(A)_i^\leq) \cap (\tptilt(A)_i^\geq).
\end{aligned}
$$
This terminology will be justified in subsection \ref{subsec g-vec}. 

Similarly we define $\tpjtilt (A)_i^\leq, \tpjtilt (A)_i^\geq$ and $\tpjtilt (A)_i^0$.
\end{dfn}

\subsection{Simplicial complexes}
In this subsection, we recall some basic terminologies related to simplicial complexes. 

Let $\Delta^0$ be a finite set. 
An {\em $($abstract$)$ simplicial complex} $\Delta$ on $\Delta^0$ is a set of subsets $F$ of $\Delta^0$ satisfying (i) if $x\in\Delta^0$, then $\{x\}\in\Delta$, (ii) if $F\in\Delta$ and $F'\subset F$, then $F'\in\Delta$. 
Elements of $\Delta$ are called \emph{faces} or \emph{simplices}.
A {\em $j$-dimensional face}, or a {\em $j$-dimensional simplex}, is an element of $\Delta$ of cardinality $j+1$.
We denote by $\Delta^j$ the subset of $\Delta$ consisting of all faces of dimension $j$. 

\begin{dfn}[See \cite{Z}] \label{f-polynomial} 
Let $\Delta\not = \emptyset$ be a simplicial complex of dimension $n-1$.
Then the sequence $(f_{-1}:=1,f_0,\cdots,f_{n-1})$
with $f_j:=\#\Delta^{j}\ (-1 \le j \le n-1)$
is called the {\em f-vector} of $\Delta$.
We set $F_\Delta(x):=\sum_{j=0}^{n} f_{j-1} x^{n-j}$, and call it the {\em $f$-polynomial} of $\Delta$. 
The polynomial  $F_\Delta(x-1)$ is called the {\em h-polynomial} of $\Delta$.
If we write $F_\Delta(x-1)=\sum_{j=0}^{n}h_j x^{n-j}$, 
then $(h_{0},h_1\cdots,h_{n})$ is called the {\em h-vector} of $\Delta$.
\end{dfn}

Note that the $f$-vector uniquely determines the $h$-vector,
and vice versa. 
We now recall basic terminologies related to lattice polytopes.
\begin{dfn} \label{dfn:triang}
\ 
\begin{enumerate}
\item {For a subset $X$ of $\bbR^n$ the {\em convex hull} $\conv(X)$ of $X$ is the smallest convex set that contains $X$.  If $X = \{v_0,\dots, v_m\}$ for some positive integer $m$, we set $\conv(v_0,\dots, v_m):= \conv(X)$.
    Note that it is given by
$$
\conv(v_0 ,\cdots, v_m) =\left\{\sum_{i=0}^m a_iv_i \right.\left|\,
\sum_{i=0}^m a_i \leq 1, a_i \in \bbR_{\ge 0} \text{ for all }i = 0,\dots, m\right\}.
$$}
A {\em lattice polytope} $\calP$ in $\bbR^n$ is the convex hull $\conv(v_0,\dots, v_m)$ of finitely many points $v_0,\dots,v_m$ in the lattice $\bbZ^n$.
    Two lattice polytopes $\calP$ and $\calQ$ are {\em lattice equivalent} 
    if they are related by an affine map.
    \item A {\em lattice simplex} $\calP$ in $\bbR^n$ is a lattice polytope such that the generating points $v_0,\cdots,v_n$ are affinely independent.
    \item A {\em unimodular simplex} is a lattice polytope that is lattice equivalent to the standard lattice simplex, i.e. the convex hull of the origin $0$ together with the standard unit vectors $e_i$ over all $1 \leq i \leq n$. 
    Equivalently, unimodular simplices are characterized as the $n$-dimensional lattice polytopes of minimal possible Euclidean volume, $1/n!$.
    Note that a unimodular simplex is a lattice simplex.
    \item A {\em subdivision} of a $n$-dimensional lattice polytope $\calP$ is a finite collection $S$ of lattice polytopes  such that
    \begin{enumerate}
        \item every face of a member of $S$ is in $S$,
        \item any two elements of $S$ intersect in a common (possibly empty) face,
        \item the union of the polytopes in $S$ is $\calP$.
    \end{enumerate}
    The maximal ($n$-dimensional) polytopes in $S$ are called {\em cells} of $S$.
    \item A {\em triangulation} is a subdivision of a polytope for which each cell of the subdivision is a simplex. 
    The triangulation is {\em unimodular} if so is every cell.
    \item For a lattice polytope $\calP$ in $\bbR^n$ and a positive integer $k$, 
    $$k\calP:=\{kt \mid t\in \calP\}$$ 
    is the {\em dilate} of $\calP$.
    Then the {\em lattice point enumerator} $E(\calP,x)$ is defined by 
    $$ 
    E(\calP,x):=\sum_{k\geq 0} \# (k\calP \cap \bbZ^n) x^k.
    $$
    It is well known that $E(\calP,x)$ is of the form 
    $$
    E(\calP,x)=\frac{h^{\ast}(\calP,x)}{(1-x)^{n+1}},
    $$
    where $h^{\ast}(\calP,x)$ is a polynomial of degree at most $n$ (see \cite{E}).
\end{enumerate}
\end{dfn}

Note that if a lattice polytope
$
\calP=\conv(v_0 ,\cdots, v_n)
$
has a triangulation $S$, 
{then a simplicial complex on the set 
$\{ v_0,\cdots,v_n \}$ is induced by $S$, which we denote by $\calP_{S}$.}
Here we recall the following classical result.

\begin{thm}{\cite[Theorem 2]{BM}, \cite[Section 9.3]{DRS}} \label{thm:BM}
If a lattice polytope $\calP$ has a unimodular triangulation $S$,
then we have 
$$
h^{\ast}(\calP,x)= E_{\calP_{S}}(x-1).
$$
In particular, the $h$-vector, and hence the $f$-vector, of $\calP_S$ is independent of the choice of unimodular triangulations. 
\end{thm}

For $v_1,\dots, v_m,\in \bbR^n\setminus\{0\}$,
we set for simplicity 
$$\conv_0(v_1,\dots, v_m):= \conv(0,v_1,\dots, v_m)
.$$
Then we need the following general fact, which will be used later.

\begin{lem}\label{lem:mutate-2-convex}
Let $v_1, \dots, v_{n-1}, w, w' \in \bbR^n\setminus\{0\}$.
Assume
\begin{equation}\label{eq:mutate-2}
w + w' = \sum_{i\in I}v_i
\end{equation}
for some non-empty $I \subseteq \{1,\dots,n-1\}$
with $\# I \le 2$.
Then
$$
\conv_0(v_1,\dots,v_{n-1}, w) \cup \conv_0( v_1,\dots,v_{n-1},w') = \conv_0(v_1,\dots,v_{n-1}, w, w')
$$
and hence the union is convex.
\end{lem}

\begin{proof}
For simplicity we set
$$
\begin{aligned}
C&:= \conv_0(v_1,\dots,v_{n-1}, w) \ \text{ and}\\
C'&:= \conv_0(v_1,\dots,v_{n-1}, w').
\end{aligned}
$$
It is obvious that $C, C' \subseteq \conv_0(v_1,\dots,v_{n-1}, w, w')$.
To show the converse inclusion, take any $v \in \conv_0(v_1,\dots,v_{n-1}, w, w')$.  Then
$$
v = \sum_{i=1}^{n-1} a_i v_i + aw + a'w'
$$
for some $a_i, a, a' \in \bbR_{\ge 0}$
with $\sum_{i=1}^{n-1} a_i + a +a' \le 1$.

{\bf Case 1.}
$\#I = 2$, say $I = \{j, k\},\ (j \ne k)$.
By the equality \eqref{eq:mutate-2}, we have
$$
\left\{
\begin{aligned}
v&= \sum_{i\ne j,k} a_iv_i + (a_j+a')v_j + (a_k + a')v_k + (a -a')w, \text{and}\\
v&= \sum_{i\ne j,k} a_iv_i + (a_j+a)v_j + (a_k + a)v_k + (a' -a)w'.
\end{aligned}
\right.
$$
In each equality note that the sum of coefficients of vectors is equal to $\sum_{l=1}^{n-1} a_l +a + a'$, which is at most 1.
Hence by looking at the positivity of coefficients we see that
$$
v \in
\begin{cases}
C & \text{if $a \ge a'$},\\
C'& \text{if $a' \ge a$}.
\end{cases}
$$
Thus we have $v \in C \cup C'$.

{\bf Case 2.}
$\#I = 1$, say $I = \{j\}$.
Again by the equality \eqref{eq:mutate-2} we have
$$
\left\{
\begin{aligned}
v&= \sum_{i\ne j} a_iv_i + (a_j+a')v_j + (a -a')w, \text{and}\\
v&= \sum_{i\ne j} a_iv_i + (a_j+a)v_j + (a' -a)w'.
\end{aligned}
\right.
$$
For the former (resp.\ the latter) equality note that the sum of coefficients of vectors is equal to $\sum_{i=1}^{n-1} a_i + a \le 1 - a' \le 1$ (resp.\ $\sum_{i=1}^{n-1} a_i + a' \le 1 - a \le 1$).
In both cases it is at most $1$.
Hence we have
$$
v \in
\begin{cases}
C & \text{if $a \ge a'$},\\
C'& \text{if $a' \ge a$}.
\end{cases}
$$
Thus we have $v \in C \cup C'$.
\end{proof}

\subsection{$g$-vectors and $g$-polytopes}\label{subsec g-vec}
In this subsection, 
we recall the notions of $g$-vectors 
and the convex hull of 2-term tilting complexes (see \cite{AIR,DIJ,H1} for the details).
We also consider the lattice polytope and 
the simplicial complex associated with 2-term tilting complexes. We keep  Notation \ref{notation} and,  in particular, $A$ is assumed to be symmetric.

Let $K_0(\Kb(\proj A))$ be the Grothendieck group of 
$\Kb(\proj A)$.
{We denote by $[P]$ the equivalence class of $P$ in $K_0(\Kb(\proj A))$ for all $P \in \Kb(\proj A)$}. 
Following \cite{DIJ}, we give the following definition. 

\begin{dfn}
We define a simplicial complex $\Delta=\Delta(A)$ on the set
$$\Delta^0:=\{[T]\ |\ T\in\indtptilt(A)\}$$
by saying that a subset $\{[T_1],\ldots,[T_j]\}$ of $\Delta^0$ is a simplex of $\Delta$ if $T_1\oplus\cdots\oplus T_j\in\tptilt (A)$.
\end{dfn}

Hence, the set of $j$-dimensional faces $\Delta^{j}$ corresponds to $\tptilt^{j+1}(A)$, 
and for the $f$-vector $(f_0,f_1,\cdots,f_{n-1})$ of $\Delta(A)$, the entry $f_{j-1}$ is equal to the number $\#\tpjtilt (A)$ for $1\leq j\leq n$. 

\begin{dfn}\label{g-vector}
For each $P \in \Kb(\proj A)$, we set
$$
g(P):= \pmat{g_1\\ \vdots\\ g_n} \in \bbZ^n
$$
if $[P]=\sum_{i=1}^n g_i[e_iA]$ in $K_0(\Kb(\proj A))$. 
Let $T:= \Ds_{i=1}^n T_i \in \ttilt{A}$ with all $T_i$ indecomposable.
Then we define the {\em convex hull} of $T$ by
\begin{equation}\label{eq:dfn-conv-hull}
\conv_0(T):= \conv(0, g(T_1), \dots, g(T_n))
\end{equation}
and the {\em $g$-polytope} of $A$ by 
\begin{equation}
\mathcal{P}(A):=\bigcup_{T \in \tiny\ttilt A}\conv_0(T).
\end{equation}
\end{dfn}

We recall the following important result, which is an analogue of classical result about tilting modules \cite{H1}.

\begin{prp}\label{volume}
\begin{enumerate}
\item Let $T,U\in \ttilt A$. If $T\ncong U$, then $\conv_0(T)$ and 
 $\conv_0(U)$ intersect only at their boundaries.
\item For any $T\in \ttilt A$, the volume of $\conv_0(T)=\frac{1}{n!}$.
\item  $\calP (A)$ is homeomorphic to an $n$-dimensional ball.
\item $\calP (A)$ is a lattice polytope in $\bbR^n$ admitting a  unimodular triangulation, 
which is given by $\conv_0(T)$ with $T\in\ttilt{A}$ as maximal cells. 
\end{enumerate}
\end{prp}

\begin{proof}
(1) follows from \cite[Theorem 1.9]{DIJ}. (2) follows from the fact that the $g$-vector of $T$ gives a $\mathbb{Z}$-basis of $\mathbb{Z}^n\cong K_0(\Kb(\proj A))$ \cite[Theorem 5.1]{AIR}. 

Since $A$ is representation-finite, the number of $\ttilt A$ is finite. 
Hence it follows by \cite[Theorem 5.4]{DIJ} that $\Delta(A)$ 
is an $(n-1)$-dimensional sphere. On the other hand the geometric realization of $\Delta(A)$ coincides with the boundary of $\calP(A)$,which contains the origin as an inner point. Hence $\calP(A)$ is an $n$-dimensional ball. 
Moreover, (4) follows from (1),(2).
\end{proof}

Note that $\Delta(A)$ is nothing but the simplicial complex induced by the unimodular triangulation of $\ttilt A$.

Next we give a geometric interpretation of Definition \ref{def upper tilt}. 
For $i \in Q_0$, we set
$$
\begin{aligned}
H^{\le }_i&:=\{(v_j)_j^n \in \bbR^n \mid v_i \le 0\},\\
H^{\geq }_i&:=\{(v_j)_j^n \in \bbR^n \mid v_i \geq0\},\\
H^{0}_i&:=\{(v_j)_j^n \in \bbR^n \mid v_i = 0\}.
\end{aligned}$$

Then we have the following lemma. 

\begin{lem}\label{upper or lower}
\begin{enumerate}
\item For each $i\in Q_0$ and $1 \le j \le n$,
we have the following equalities.
$$
\begin{aligned}
\tpjtilt(A)_i^\leq &= \{ T\in \tpjtilt(A) \mid g(T)\in H^{\le }_i \}, \\
\tpjtilt(A)_i^\geq &= \{ T\in \tpjtilt(A) \mid g(T)\in H^{\ge }_i \}, \\
\tpjtilt(A)_i^0 &= \{ T\in \tpjtilt(A) \mid g(T)\in H^{0}_i \}.
\end{aligned}$$

\item For each $T \in \ttilt{A}$ and $i\in Q_0$, we have
$\conv_0(T) \subseteq H^{\ge}_i$ or $\conv_0(T) \subseteq H^{\le }_i$.

\item For any $i\in Q_0$, we have 
$$\calP(A)\cap H^{\le }_i=\bigcup_{T \in \tiny\ttilt A_i^\leq} \conv_0(T),\ \ \calP(A)\cap H^{\ge }_i=\bigcup_{T \in \tiny\ttilt A_i^\geq} \conv_0(T). $$
In particular, the unimodular triangulation of $P(A)$ induces the one on $P(A) \cap H_i^{\le}$ and on
$P(A) \cap H_i^{\ge}$ for each $i \in Q_0$.

\item Let $S$ (resp. $S'$) be the unimodular triangulation of $\calP(A)\cap H^{\le }_i$ (resp. $\calP(A)\cap H^{\ge }_i$). 
The number of $(j+1)$-dimensional faces of the simplicial complex 
$(\calP(A)\cap H^{\le }_i)_S$ (resp. $(\calP(A)\cap H^{\ge }_i)_{S'}$) coincides with $\# \tpjtilt(A_G)_i^\leq$ (resp. $\# \tpjtilt(A_G)_i^\geq$).
\end{enumerate}
\end{lem}

\begin{proof}(1) is immediate from the definition and (2)
follows by Lemma \ref{no common}. 

(3) follows from (1), (2) and Proposition \ref{volume}. 
Finally (4) follows from (3) and Theorem \ref{thm:BM}.
\end{proof}

Now we give some examples.

\begin{exm}\label{exam1}
(1) Let $G$ be a Brauer tree with two edges and $A_G$ the Brauer tree algebra of $G$.
Then $A_G$ is an algebra defined by the following quiver with relations:
$$
\xymatrix{
1 \ar@/^/[r]^{\al_1} & 2 \ar@/^/[l]^{\al_2}
}, \al_1\al_2\al_1 = 0, \al_2\al_1\al_2 = 0.
$$
Then the shaded region of the following picture denotes $\mathcal{P}(A_G)$, which can be decomposed into 6 unimodular simplices. 
Moreover 6 vertices stand for 0-dimensional faces, and 6 line segments on the boundary of $\calP(A_G)$
stand for 1-dimensional faces of $\Delta(A_G)$.
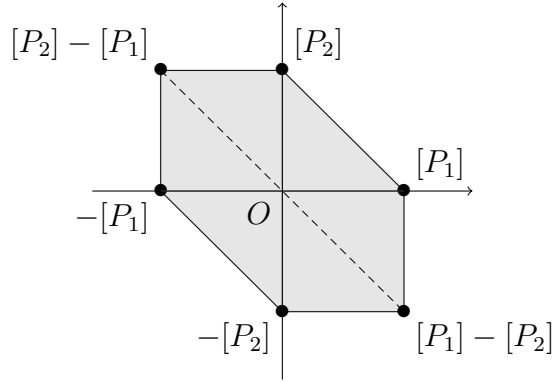
\[
\begin{tikzpicture}
\filldraw[draw=gray,thick,fill=black,opacity=.1] (-1.6,1.6)--(0,1.6)--(1.6,0)--(1.6,-1.6)--(0,-1.6)--(-1.6,0)--cycle;
\draw[draw=black,ultra thin] (-1.6,1.6)--(0,1.6)--(1.6,0)--(1.6,-1.6)--(0,-1.6)--(-1.6,0)--cycle;
\draw[draw=black,ultra thin,densely dashed] (-1.6,1.6)--(1.6,-1.6);
\draw (0,0) node[below left] {$O$};
\draw (-1.6,1.6) node {$\bullet$};
\draw (-1.6,1.6) node[above left] {$[P_2]-[P_1]$};
\draw (0,1.6) node {$\bullet$};
\draw (0,1.6) node[above right] {$[P_2]$};
\draw (1.6,0) node {$\bullet$};
\draw (1.6,0) node[above right] {$[P_1]$};
\draw (1.6,-1.6) node {$\bullet$};
\draw (1.6,-1.6) node[below right] {$[P_1]-[P_2]$};
\draw (0,-1.6) node {$\bullet$};
\draw (0,-1.6) node[below left] {$-[P_2]$};
\draw (-1.6,0) node {$\bullet$};
\draw (-1.6,0) node[below left] {$-[P_1]$};

\draw[draw=gray,ultra thin] (0,0)--(0,-1.6);
\draw[draw=gray,ultra thin] (0,0)--(1.6,0);
\draw[draw=gray,ultra thin] (0,0)--(-1.6,0);
\draw[draw=gray,ultra thin] (0,0)--(0,1.6);
\draw[draw=black,thin,->] (-2.5,0)--(2.5,0);
\draw[draw=black,thin,->] (0,-2.5)--(0,2.5);
\end{tikzpicture}\]

Thus we have the $f$-vector $(f_{-1},f_0,f_1)=(1,6,6)$ and the $h$-vector $(h_{0},h_1,h_2)=(1,4,1)$.

(2) Let $G$ be a linear tree with three edges and $A_G$ the Brauer tree algebra of $G$. 
Then $\calP(A_G)$ (and $\Delta(A_G)$) are described as follows. 

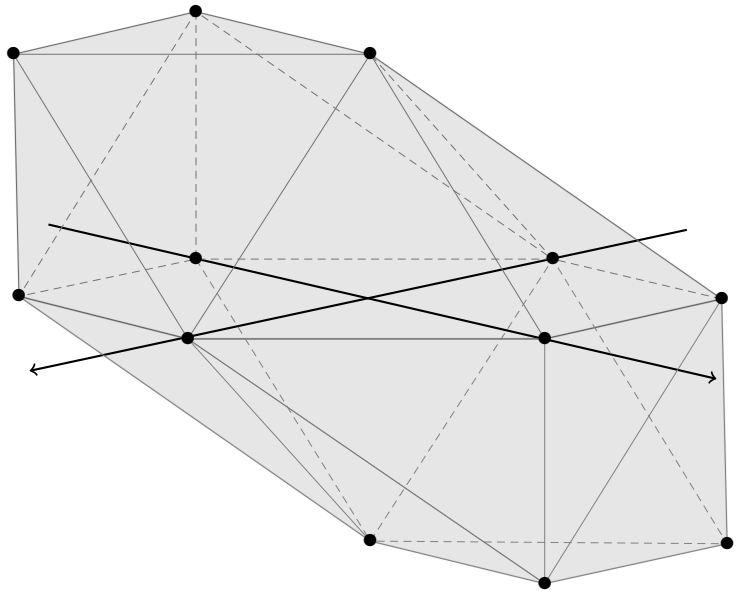
\[
\begin{tikzpicture}
\draw[draw=black,->,thick] (118bp,26bp)--(-127bp,-27bp);
\draw[draw=black,->,thick] (-120bp,28bp)--(129bp,-30bp);
\draw[draw=gray,ultra thin,densely dashed] (131bp,0bp)--(68bp,15bp)--(-65bp,15bp)--(-131bp,1bp)--(-65bp,108bp)--(68bp,15bp)--(0bp,92bp);
\draw[draw=gray,ultra thin,densely dashed] (-65bp,108bp)--(-65bp,15bp);
\draw[draw=gray,ultra thin] (0bp,92bp)--(-65bp,108bp)--(-133bp,92bp)--(-131bp,1bp)--(-68bp,-15bp)--(65bp,-15bp)--(131bp,0bp)--(0bp,92bp)--(-133bp,92bp)--(-68bp,-15bp)--(0bp,92bp)--(65bp,-15bp);
\filldraw[draw=black,thick,fill=black,opacity=.1] (-131bp,1bp)--(-68bp,-15bp)--(65bp,-15bp)--(131bp,0bp)--(0bp,92bp)--(-65bp,108bp)--(-133bp,92bp)--cycle;
\filldraw[draw=black,thick,fill=black,opacity=.1] (-131bp,1bp)--(-68bp,-15bp)--(65bp,-15bp)--(131bp,0bp)--(133bp,-92bp)--(65bp,-107bp)--(0bp,-91bp)--cycle;
\draw[draw=gray,ultra thin] (-131bp,1bp)--(0bp,-91bp)--(-68bp,-15bp)--(65bp,-107bp)--(131bp,0bp)--(133bp,-92bp)--(65bp,-107bp)--(0bp,-91bp);
\draw[draw=gray,ultra thin] (-68bp,-15bp)--(65bp,-107bp)--(65bp,-15bp);
\draw[draw=gray,ultra thin,densely dashed] (-65bp,15bp)--(0bp,-91bp)--(68bp,15bp)--(133bp,-92bp)--(0bp,-91bp);
\draw (-68bp,-15bp) node (P1) {$\bullet$};
\draw (65bp,-15bp) node (P2) {$\bullet$};
\draw (0bp,92bp) node (P3) {$\bullet$};
\draw (68bp,15bp) node (Q1) {$\bullet$};
\draw (-65bp,15bp) node (Q2) {$\bullet$};
\draw (-131bp,1bp) node (S1) {$\bullet$};
\draw (131bp,0bp) node (T1) {$\bullet$};
\draw (-133bp,92bp) node (T2) {$\bullet$};
\draw (-65bp,108bp) node (S3) {$\bullet$};
\draw (0bp,-91bp) node (Q3) {$\bullet$};
\draw (65bp,-107bp) node (T3) {$\bullet$};
\draw (133bp,-92bp) node (S2) {$\bullet$};
\end{tikzpicture}
\]

We remark that the partial order of $\ttilt A_G$ is entirely determined by the cones of $g$-vectors \cite[Theorem 6.12, Corollary 6.13]{DIJ}. 
In this case, the Hasse quiver (= mutation quiver)  of $\ttilt A_G$ can be recovered as follows.
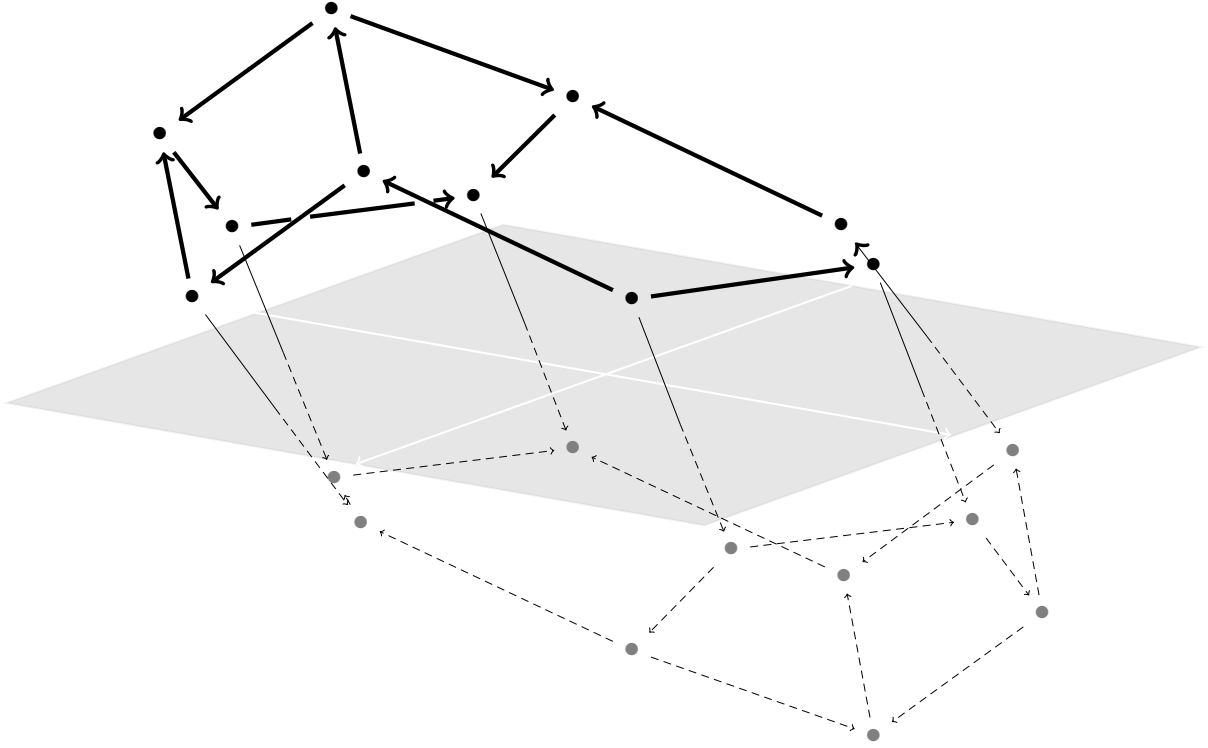
\[
\begin{tikzpicture}
\filldraw[draw=gray,thick,fill=black,opacity=.1] (-37bp,56bp)--(223bp,10bp)--(38bp,-57bp)--(-222bp,-11bp)--cycle;
\draw[draw=white,->,thick] (93bp,33bp) to (-92bp,-34bp);
\draw[draw=white,->,thick] (-130bp,23bp) to (130bp,-23bp);
\draw (11bp,28bp) node (G1) {$\bullet$};
\draw (101bp,41bp) node (G2) {$\bullet$};
\draw (-89bp,76bp) node (G3) {$\bullet$};
\draw (48bp,-66bp) node[gray] (G4) {$\bullet$};
\draw (89bp,56bp) node (G5) {$\bullet$};
\draw (138bp,-55bp) node[gray] (G6) {$\bullet$};
\draw (-101bp,137bp) node (G7) {$\bullet$};
\draw (-153bp,29bp) node (G8) {$\bullet$};
\draw (11bp,-104bp) node[gray] (G9) {$\bullet$};
\draw (153bp,-29bp) node[gray] (G10) {$\bullet$};
\draw (-11bp,104bp) node (G11) {$\bullet$};
\draw (164bp,-90bp) node[gray] (G12) {$\bullet$};
\draw (-165bp,90bp) node (G13) {$\bullet$};
\draw (-90bp,-56bp) node[gray] (G14) {$\bullet$};
\draw (101bp,-136bp) node[gray] (G15) {$\bullet$};
\draw (90bp,-76bp) node[gray] (G16) {$\bullet$};
\draw (-48bp,67bp) node (G17) {$\bullet$};
\draw (-138bp,55bp) node (G18) {$\bullet$};
\draw (-100bp,-39bp) node[gray] (G19) {$\bullet$};
\draw (-11bp,-28bp) node[gray] (G20) {$\bullet$};
\draw[draw=black,->,ultra thick] (G1) to (G2);
\draw[draw=black,->,ultra thick] (G1) to (G3);
\draw[draw=black,ultra thin] (G1) to (29bp,-19bp);
\draw[draw=black,->,ultra thin,densely dashed] (29bp,-19bp) to (G4);
\draw[draw=black,->,ultra thick] (G2) to (G5);
\draw[draw=black,ultra thin] (G2) to (119bp,-6bp);
\draw[draw=black,->,ultra thin,densely dashed] (119bp,-6bp) to (G6);
\draw[draw=black,->,ultra thick] (G3) to (G7);
\draw[draw=black,->,ultra thick] (G3) to (G8);
\draw[draw=black,->,ultra thin,densely dashed] (G4) to (G6);
\draw[draw=black,->,ultra thin,densely dashed] (G4) to (G9);
\draw[draw=black,ultra thin] (G5) to (121bp,14bp);
\draw[draw=black,->,ultra thin,densely dashed] (121bp,14bp) to (G10);
\draw[draw=black,->,ultra thick] (G5) to (G11);
\draw[draw=black,->,ultra thin,densely dashed] (G6) to (G12);
\draw[draw=black,->,ultra thick] (G7) to (G11);
\draw[draw=black,->,ultra thick] (G7) to (G13);
\draw[draw=black,->,ultra thick] (G8) to (G13);
\draw[draw=black,ultra thin] (G8) to (-122bp,-13bp);
\draw[draw=black,->,ultra thin,densely dashed] (-122bp,-13bp) to (G14);
\draw[draw=black,->,ultra thin,densely dashed] (G9) to (G14);
\draw[draw=black,->,ultra thin,densely dashed] (G9) to (G15);
\draw[draw=black,->,ultra thin,densely dashed] (G10) to (G16);
\draw[draw=black,->,ultra thick] (G11) to (G17);
\draw[draw=black,->,ultra thin,densely dashed] (G12) to (G10);
\draw[draw=black,->,ultra thin,densely dashed] (G12) to (G15);
\draw[draw=black,->,ultra thick] (G13) to (G18);
\draw[draw=black,->,ultra thin,densely dashed] (G14) to (G19);
\draw[draw=black,->,ultra thin,densely dashed] (G15) to (G16);
\draw[draw=black,->,ultra thin,densely dashed] (G16) to (G20);
\draw[draw=black,ultra thin] (G17) to (-29bp,19bp);
\draw[draw=black,->,ultra thin,densely dashed] (-29bp,19bp) to (G20);
\draw[draw=black,ultra thick] (G18) to (-116bp,58bp);
\draw[draw=black,ultra thick] (-109bp,59bp) to (-70bp,64bp);
\draw[draw=black,->,ultra thick] (-63bp,65bp) to (G17);
\draw[draw=black,ultra thin] (G18) to (-119bp,8bp);
\draw[draw=black,->,ultra thin,densely dashed] (-119bp,8bp) to (G19);
\draw[draw=black,->,ultra thin,densely dashed] (G19) to (G20);
\end{tikzpicture}
\]

\end{exm}

\subsection{Convexity}
In this subsection, we characterize when the polytope $\calP(A)$ is convex. 
Throughout this subsection, let $A$ be 
a finite dimensional symmetric algebra such that $\#\ttilt(A)$ is finite (this is equivalent to saying that $A$ is  \emph{$\tau$-tilting finite} \cite{DIJ}). 
{Assume that $|A|=n$, that is, $\Delta(A)$ is a simplicial complex 
of dimension $n-1$.
In this case, every $(n-2)$-simplex of $\Delta(A)$ is contained in exactly two maximal simplices \cite[Theorem 5.2]{DIJ}. 
Equivalently, there are exactly two 2-term tilting complexes $T,T'\in\ttilt(A)$ such that $T$ and $T'$ share $n-1$ indecomposable direct summands (this is also equivalent to saying that $T$ is obtained from $T'$ by mutation \cite[Corollary 3.8]{AIR}).
In this case, we call the two polytopes $\conv_0(T)$ and $\conv_0(T')$ \emph{adjacent}.
Note that every Brauer tree algebra is symmetric and $\tau$-tilting finite.}





First we prepare a basic property about the convexity of $\calP(A)$. 
The proof below was informed us by Osamu Iyama, for which we are thankful.




\begin{prp}\label{prp:2-conv-whole-conv}

If the union of any two adjacent polytopes $\conv_0(T)$ and $\conv_0(T')$ {with} $T,T'\in\ttilt A$ is convex, 
then the $g$-polytope $\calP(A)$ is convex. 
\end{prp}

\begin{proof}
We let $\calP:=\calP(A)$ for simplicity.
Let $H$ be a 2-dimensional plane containing the origin and
consider the intersection $\calS:= \mathcal{P} \cap H$. 
Since the convexity is a condition of a line segment connecting two points, 
$\calP$ is convex if and only if $\calS$ is convex for any $H$. 
Without loss of generality, we may assume that $\calS$ does not contain any non-zero vertex of $\calP$ because we can move $H$ slightly in order not to contain any vertex of $\calP$ keeping the property that $\calS$ is not convex. 
Hence, for each convex set $\conv_0(T)$, the intersection of $\conv_0(T)$ and $H$ is a triangle containing the origin as its vertex. Thus, by Proposition \ref{volume} (3),  $\calS$ must be a polygon. 
Then the assumption implies that the union of any two adjacent triangles of $\calS$ is convex, which shows that each interior angle of the polygon $\calS$ is at most $180^\circ$. Hence $\calS$ is convex.
\end{proof}


\begin{exm}
The following picture shows the situation in Proposition \ref{prp:2-conv-whole-conv} in the case of the Brauer tree algebra $A_G$ with a linear tree $G$ having 3 edges. 

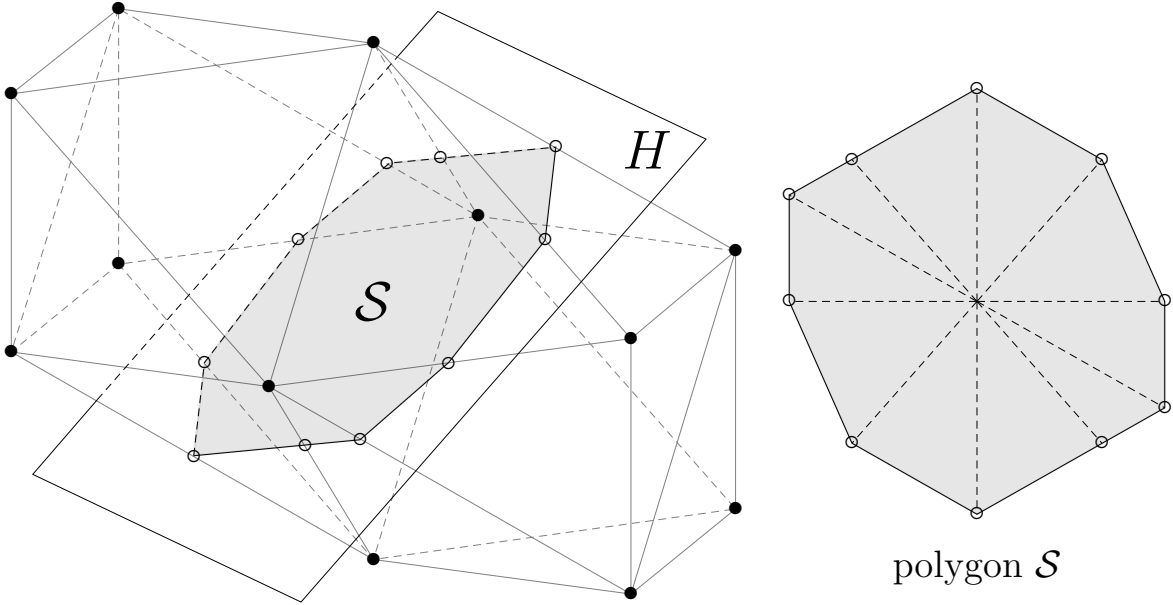
\[
\begin{tikzpicture}
\draw[draw=gray,ultra thin,densely dashed] (-185bp,110bp)--(-225bp,-19bp)--(-185bp,14bp)--(-185bp,110bp)--(-51bp,32bp)--(-185bp,14bp)--(-90bp,-97bp)--(-51bp,32bp)--(-90bp,97bp);
\draw[draw=gray,ultra thin] (-90bp,97bp)--(-185bp,110bp)--(-225bp,78bp)--(-225bp,-19bp)--(-129bp,-32bp)--(-225bp,78bp)--(-90bp,97bp)--(-129bp,-32bp)--(6bp,-14bp)--(-90bp,97bp)--(45bp,19bp)--(6bp,-14bp);
\draw[draw=gray,ultra thin] (-225bp,-19bp)--(-90bp,-97bp)--(-129bp,-32bp)--(6bp,-110bp)--(45bp,19bp)--(45bp,-78bp)--(6bp,-110bp)--(-90bp,-97bp);
\draw[draw=gray,ultra thin] (6bp,-14bp)--(6bp,-110bp);
\draw[draw=gray,ultra thin,densely dashed] (45bp,19bp)--(-51bp,32bp)--(45bp,-78bp)--(-90bp,-97bp);
\draw[draw=black,ultra thin] (-81bp,92bp)--(-66bp,109bp)--(34bp,61bp)--(-117bp,-113bp)--(-217bp,-65bp)--(-193bp,-37bp);
\draw[draw=black,ultra thin,densely dashed] (-193bp,-37bp)--(-81bp,92bp);
\filldraw[draw=black,thick,fill=black,opacity=.1] (-22bp,58bp)--(-26bp,23bp)--(-62bp,-23bp)--(-95bp,-52bp)--(-157bp,-58bp)--(-153bp,-23bp)--(-118bp,23bp)--(-84bp,52bp)--cycle;
\draw[draw=black,ultra thin] (-22bp,58bp)--(-26bp,23bp)--(-62bp,-23bp)--(-95bp,-52bp)--(-157bp,-58bp);
\draw[draw=black,ultra thin,densely dashed] (-157bp,-58bp)--(-153bp,-23bp)--(-118bp,23bp)--(-84bp,52bp)--(-22bp,58bp);
\draw (-129bp,-32bp) node (P1) {$\bullet$};
\draw (6bp,-14bp) node (P2) {$\bullet$};
\draw (-90bp,97bp) node (P3) {$\bullet$};
\draw (-51bp,32bp) node (Q1) {$\bullet$};
\draw (-185bp,14bp) node (Q2) {$\bullet$};
\draw (-225bp,-19bp) node (S1) {$\bullet$};
\draw (45bp,19bp) node (T1) {$\bullet$};
\draw (-225bp,78bp) node (T2) {$\bullet$};
\draw (-185bp,110bp) node (S3) {$\bullet$};
\draw (-90bp,-97bp) node (Q3) {$\bullet$};
\draw (6bp,-110bp) node (T3) {$\bullet$};
\draw (45bp,-78bp) node (S2) {$\bullet$};
\draw (-22bp,58bp) node (R1) {$\circ$};
\draw (-26bp,23bp) node (R2) {$\circ$};
\draw (-62bp,-23.5bp) node (R3) {$\circ$};
\draw (-95bp,-52bp) node (R4) {$\circ$};
\draw (-157bp,-58.5bp) node (R5) {$\circ$};
\draw (-153bp,-23bp) node (R6) {$\circ$};
\draw (-118bp,23bp) node (R7) {$\circ$};
\draw (-85bp,51.5bp) node (R8) {$\circ$};
\draw (-115.5bp,-54.5bp) node (U) {$\circ$};
\draw (-65bp,54bp) node (V) {$\circ$};

\draw (12bp,57bp) node (H) {{\LARGE $H$}};
\draw (-90bp,0bp) node (S) {{\LARGE $\calS$}};
\draw (135bp,-100bp) node (T) {{\large polygon $\calS$}};

\draw[draw=black,ultra thin] (135bp,80bp)--(181.6bp,53.3bp)--(205bp,0bp)--(205bp,-40bp)--(135bp,-80bp)--(88.4bp,-53.3bp)--(65bp,0bp)--(65bp,40bp)--cycle;
\draw[draw=black,ultra thin,densely dashed] (135bp,0bp)--(135bp,80bp);
\draw[draw=black,ultra thin,densely dashed] (135bp,0bp)--(181.6bp,53.3bp);
\draw[draw=black,ultra thin,densely dashed] (135bp,0bp)--(205bp,0bp);
\draw[draw=black,ultra thin,densely dashed] (135bp,0bp)--(205bp,-40bp);
\draw[draw=black,ultra thin,densely dashed] (135bp,0bp)--(135bp,-80bp);
\draw[draw=black,ultra thin,densely dashed] (135bp,0bp)--(88.4bp,-53.3bp);
\draw[draw=black,ultra thin,densely dashed] (135bp,0bp)--(65bp,0bp);
\draw[draw=black,ultra thin,densely dashed] (135bp,0bp)--(65bp,40bp);
\draw[draw=black,ultra thin,densely dashed] (135bp,0bp)--(181.6bp,-53.3bp);
\draw[draw=black,ultra thin,densely dashed] (135bp,0bp)--(88.4bp,53.3bp);
\filldraw[draw=black,thick,fill=black,opacity=.1] (135bp,80bp)--(181.6bp,53.3bp)--(205bp,0bp)--(205bp,-40bp)--(135bp,-80bp)--(88.4bp,-53.3bp)--(65bp,0bp)--(65bp,40bp)--cycle;

\draw (135bp,80bp) node (R1) {$\circ$};
\draw (181.6bp,53.3bp) node (R2) {$\circ$};
\draw (205bp,0bp) node (R3) {$\circ$};
\draw (205bp,-40bp) node (R4) {$\circ$};
\draw (135bp,-80bp) node (R5) {$\circ$};
\draw (88.4bp,-53.3bp) node (R6) {$\circ$};
\draw (65bp,0bp) node (R7) {$\circ$};
\draw (65bp,40bp) node (R8) {$\circ$};
\draw (181.6bp,-53.3bp) node (U) {$\circ$};
\draw (88.4bp,53.3bp) node (V) {$\circ$};

\end{tikzpicture}
\]

\end{exm}





Next we introduce the following terminology (we refer to \cite{AI} for the notion of mutation). 

\begin{dfn}\label{dfn:left-mutation}
Let $T\in\ttilt A$ and $T= \Ds_{i=1}^n T_i$ be a decomposition into indecomposable complexes $T_i$
($1 \le i \le n$). 
\begin{enumerate}
\item For the left mutation sequence  
$$\xymatrix{
T_i \ar[r]^(0.4){} & \Ds_{k \in I} T_k \ar[r] & T_i' \ar[r]&T_i[1]
}$$
 such that $T_i'\in\tptilt A$ (or equivalently,  $T_i'\oplus \Ds_{j\neq i} T_j\in\ttilt A$), if $\#I = |\Ds_{k\in I} T_k |\leq 2$, then we say that \emph{$T_i$ admits at most two indecomposable left approximation.
} 
\item We say that \emph{$T$ admits at most two indecomposable left  approximation} if so does any indecomposable direct summand of $T$. 
\item We say that \emph{$\ttilt A$ admits at most two indecomposable left  approximation} if so does any $T\in\ttilt A$. 
\end{enumerate}
\end{dfn}

We remark that if $\ttilt A$ admits at most two indecomposable left approximation, then $\ttilt A$ also admits at most two indecomposable right approximation because left and right mutation are invertible to each other.

The following proposition gives a sufficient condition for the polytope 
$\calP(A)$ to be convex, which was inspired by a similar argument used in \cite{H2}. 

\begin{prp}\label{union will be convex}
If $\ttilt A$ admits at most two indecomposable left approximation, then 
$\calP(A)$ is convex. 
\end{prp}

\begin{proof}
Let $T\in\ttilt A$ and $T= \Ds_{i=1}^n T_i$ be a decomposition into indecomposable complexes $T_i$. 
Take the left mutation sequence  
$$\xymatrix{
T_i \ar[r]^(0.4){} & \Ds_{} T_k \ar[r] & T_i' \ar[r]&T_i[1]
}$$ such that $T_i'\in\tptilt A$.    
Then, by our assumption, we have $g(T_i)+g(T_i')=\sum_{k\in I}g(T_k)$ with $\# I\leq 2$.
Therefore, 
we apply Lemma \ref{lem:mutate-2-convex} by taking 
$w=g(T_i),w'=g(T_i')$ and
hence $\conv_0(T)\cup\conv_0(T')$ is convex, where $T':= T_i'\oplus \Ds_{j\neq i} T_j\in\ttilt A$. 
Moreover, since $\#\ttilt(A)$ is finite, $\ttilt A$ are transitive by the action of mutation \cite[Corollary 3.10]{AIR}. 
Thus, any two adjacent convex hulls always becomes convex and hence Proposition \ref{prp:2-conv-whole-conv} implies that $\calP(A)$ is convex.
\end{proof}






\section{Symmetry of the polytopes}
In this section, we study geometric properties of $\calP(A)$ for a Brauer tree algebra $A$. We show the convexity and symmetry of $\calP(A)$.
This fact gives the correspondence of the $f$-vector of the upper half part and lower half part of $\Delta(A)$ divided by $H_i^0$ for any $i\in Q_0$.

The aim of this section is to show the following result.

\begin{thm} \label{thm:main1}
Let $G$ be a Brauer tree with multiplicity 1 and $A_G$ the Brauer tree algebra of $G$. 
Then, for any $i\in Q_{G,0}$ and $1 \le j \le |A_G|$, we have 
$$\# \tpjtilt(A_G)_i^\leq=\# \tpjtilt(A_G)_i^\geq.$$
\end{thm}
We keep Notation \ref{notation}, and 
in the sequel, we set $A:= A_G$ and $n:= |G| = |A|$.

First, following \cite{AAC}, we use the following definition.
\newcommand{\sign}{\operatorname{sign}}

\begin{dfn}
An {\em alternating signed walk} of $G$ is a pair $(w, s)$
of a walk $w=a_{1}a_{2}\cdots a_{\ell}$ of $G$ and a map
$s\colon \{a_{1}, a_{2}, \cdots, a_{\ell}\} \to \{1, -1\}$
such that $s(a_{k})=-s(a_{{k+1}})$ for all $1\leq k\leq \ell-1$.
When there seems to be no confusion, we just denote it by $w$
and denote $s$ by $\sign$.
\end{dfn}

\begin{exm}\label{exam signed walk}
The line along the tree with $+, -$ signs shown below represents an alternating signed walk $(w,s)$ in $G$.
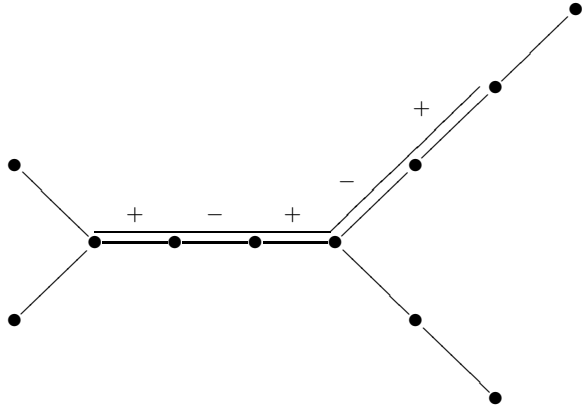
$$
\xymatrix@M=0pt{
&&&&&&&\bullet\\
&&&&&&\bullet\\
\bullet&&&&&\bullet\\
&\bullet&\bullet&\bullet&\bullet\\
\bullet&&&&&\bullet\\
&&&&&&\bullet
\ar@{-}"3,1";"4,2"
\ar@{-}"4,2";"4,3"
\ar@{-}"4,3";"4,4"
\ar@{-}"4,4";"4,5"
\ar@{-}"4,5";"3,6"
\ar@{-}"3,6";"2,7"
\ar@{-}"2,7";"1,8"
\ar@{-}"5,1";"4,2"
\ar@{-}"4,5";"5,6"
\ar@{-}"5,6";"6,7"
\ar@{-}"4,2"+<0pt,4pt>;"4,3"+<0pt,4pt>^+
\ar@{-}"4,3"+<0pt,4pt>;"4,4"+<0pt,4pt>^{-}
\ar@{-}"4,4"+<0pt,4pt>;"4,5"+<-2pt,4pt>^+
\ar@{-}"4,5"+<-2pt,4pt>;"3,6"+<-6pt,0pt>^{-}
\ar@{-}"3,6"+<-6pt,0pt>;"2,7"+<-6pt,0pt>^+
}
$$

\end{exm}



The following lemma is immediate from the definition of Brauer tree algebras. 


\begin{lem}\label{hom-dim}
Let $a,b$ be distinct edges of $G$.
Then we have 
$$
\Hom_A(P_a, P_b) =
\begin{cases}
\k \la_{p(a,b)} & \text{if $a$ and $b$ are connected to a common vertex in $G$,}\\
0 & \text{otherwise,}
\end{cases}
$$
where $p(a,b)$ is the shortest path from $a$ to $b$ in the quiver $Q_G$ and $\la_p$ denotes the left multiplication by the image of $p$ in $A$ for each path $p$ in $Q_G$.
%
\end{lem}

From Lemma \ref{hom-dim}, we give the following description by \cite[Theorem 4.6]{AAC}.

\begin{prp}\label{alter walk}
We have a bijection
$$\{\textnormal{Alternating signed walks of $G$}\}\longrightarrow\ \indtptilt A.$$
The map is given as follows. 
For each alternating signed walk $(w, s)$ with $w=a_1a_2\dots a_\ell$ 
define 
$$
P(w,s):= (\cdots \to 0\to \Ds_{s(a_i)=-1}P_{a_i} \ya{\ g\ } \Ds_{s(a_j)=1}P_{a_j} \to 0 \to \cdots),
$$
where $g$ is expressed by the matrix with the $(j,i)$-entries $\la_{p(a_i,a_j)}$ given in {\rm Lemma \ref{hom-dim}}.
\end{prp}

\begin{rmk}
We can also explain Proposition \ref{alter walk} as follows. First, indecomposable $\tau$-rigid $A$-modules are easily classified (for example \cite{AZ}). Second, there is a bijective map from the set of indecomposable $\tau$-rigid $A$-modules to $\indtptilt A$, which is given by taking the minimal projective presentation \cite[Theorem 3.2]{AIR}. 
Since $\tau$-rigid $A$-modules are given as string modules, we get $P(w,s)$ above by the description of minimal projective presentations of string modules given by \cite{WW}.

\end{rmk}

\begin{exm}
(1) Let $(w,s), (w',s')$ be alternating signed walks in $G$ with $w=a_{1}a_2\cdots a_\ell$ and 
$w'=a'_1a'_2\cdots a'_\ell$.
If $s(a_1)=-1= s(a_\ell)$ and $s(a'_1)=1,s(a'_\ell)=-1$, then
$P(w,s)$ and $P(w',s')$ take the following forms respectively.
$$
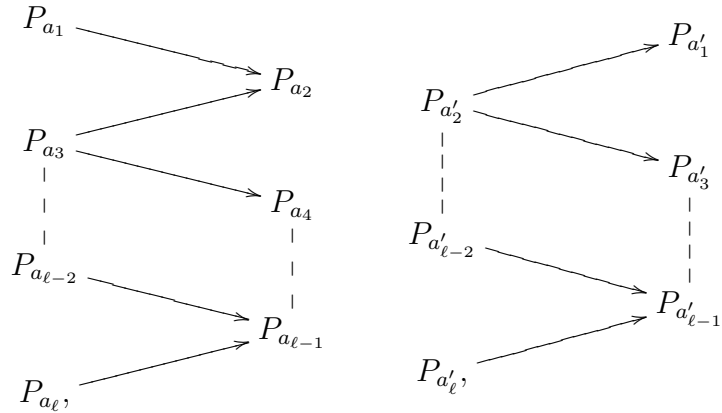
\vcenter{
\xymatrix@R=5pt@C=60pt{
P_{a_1} \ar[rd] & \\ &  P_{a_2}\\
P_{a_3} \ar@{--}[dd]\ar[ru]\ar[rd]&\\
 & P_{a_4} \ar@{--}[dd]\\
P_{a_{\ell-2}}\ar[rd] &  \\
 & P_{a_{\ell-1}} \\
P_{a_\ell}, \ar[ru] &
}
}
\qquad
\vcenter{
\xymatrix@R=5pt@C=60pt{
 &  P_{a'_1}\\
P_{a'_2} \ar@{--}[dd]\ar[ru]\ar[rd]&\\
 & P_{a'_3} \ar@{--}[dd]\\
P_{a'_{\ell-2}}\ar[rd] &  \\
 & P_{a'_{\ell-1}} \\
P_{a'_\ell}, \ar[ru] &
}
}
$$
Here, an arrow $P_a \to P_b$ represents the map $\la_{p(a,b)}$.

(2) Let $G$ be the following Brauer tree.
$$
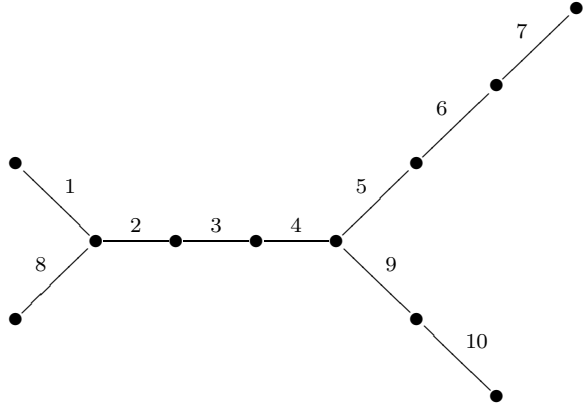
\xymatrix@M=0pt{
&&&&&&&\bullet\\
&&&&&&\bullet\\
\bullet&&&&&\bullet\\
&\bullet&\bullet&\bullet&\bullet\\
\bullet&&&&&\bullet\\
&&&&&&\bullet
\ar@{-}"3,1";"4,2"^1
\ar@{-}"4,2";"4,3"^2
\ar@{-}"4,3";"4,4"^{3}
\ar@{-}"4,4";"4,5"^4
\ar@{-}"4,5";"3,6"^{5}
\ar@{-}"3,6";"2,7"^6
\ar@{-}"2,7";"1,8"^7
\ar@{-}"5,1";"4,2"^8
\ar@{-}"4,5";"5,6"^9
\ar@{-}"5,6";"6,7"^{10}
}
$$
Then the alternating signed walk $(w, s)$ given in Example \ref{exam signed walk} corresponds to the following indecomposable pretilting complex $P(w,s)$
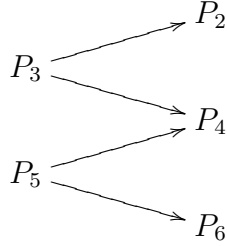
$$\xymatrix@R=3pt@C=50 pt{
 &  P_{2}\\
P_{3} \ar[rd] \ar[ru] & \\
& P_{4} \\
P_{5}\ar[rd] \ar[ru]&\\
 & P_{6}\\ }$$
\end{exm}

From Proposition \ref{alter walk}, we obtain the following result.

\begin{lem}\label{divide half}
\begin{enumerate}
\item 
For any $T\in\indtptilt A$, 
there exists a unique $T^\vee \in\indtptilt A$
such that $g(T^\vee) = -g(T)$. 

\item Let $i \in Q_0$, and let
$\{M_1, \dots, M_\ell\}$ $($resp.\ $\{N_1, \dots, N_m\})$
be the set of $\indtptilt (A)^\leq_i$ (resp. $\indtptilt (A)^\geq_i$).
Then we have $\ell=m$ and $$\{g(N_1), \dots, g(N_\ell)\}=\{-g(M_1^\vee), \dots, -g(M_\ell^\vee)\}.$$
\end{enumerate}
\end{lem}

\begin{proof}
(1) By Proposition \ref{alter walk}, we have an alternating signed walk $(w,s)$ such that $T=P(w,s)$. 
Then 
we can take 
$T^\vee:= P(w,-s)$. (2) follows from (1).
\end{proof}

%
%
%
%
%

The following lemma can be regarded as a special case of \cite[Proposition 5.11]{AAC}. For the convenience of the reader, we give a proof here. 

\begin{lem}\label{lem:tau-rigid-connected}
Let $(w, s), (w', s')$ be alternating signed walks
such that $P(w,s) \ds P(w', s')$ is pretilting.
Then
\begin{enumerate}
\item
$s(a) = s'(a)$ for all $a \in w_1 \cap w'_1$; and
\item
If $w_1 \cap w'_1 = \emptyset$, and there is a walk 
$\xymatrix@M=0pt@C=15pt{
\bullet\ar@{-}[r]^{a}&\bullet \ar@{-}[r]^{b} &\bullet}$
of length $2$ such that $a \in w_1, b \in w'_1$ then
$s(a) = s'(b)$.
\end{enumerate}
\end{lem}

\begin{proof}
Although the statement (1) follows from Lemma \ref{no common}, we also give a proof of this because it is possible to prove both statements by similar arguments. 
We set $w = a_1a_2\cdots a_m, w' = b_1b_2\cdots b_n$, and
$T:= P(s, w) \ds P(s', w')$.
Complexes $P(s, w)$ and $P(s', w')$ have the following forms:
$$
\begin{aligned}
P(s, w)&: \cdots \to 0 \to \Ds_{i\in I_1}P_{a_i} \ya{u} \Ds_{i\in I_0}P_{a_i}
\to 0 \to \cdots\\
P(s', w')&: \cdots \to 0 \to\Ds_{j\in J_1}P_{b_j} \ya{v} \Ds_{j\in J_0}P_{b_j}
\to 0 \to \cdots,
\end{aligned}
$$
where $I_0:= \{i \mid s(a_i) = 1\}, I_1:= \{i \mid s(a_i) = -1\}$
and $J_0:= \{j \mid s'(b_i) = 1\}, J_1:= \{j \mid s'(b_i) = -1\}$.

{\bf (1).}
Assume that $s(a) \ne s'(a)$ for some $a \in w_1 \cap w'_1$. Then $a = a_g = b_h$
for some $1 \le g \le m$ and $1 \le h \le n$, and
we have either (a) $g \in I_0, h \in J_1$
or (b) $g \in I_1, h \in J_0$.
Then we can take a nonzero map $f \in \Hom_A(P_a, P_a)$ as the composite:
$P_a \twoheadrightarrow \top P_a \iso \soc P_a \hookrightarrow P_a$,
which is a left multiplication of an element $x$ of the simple socle of $P_a$.
Thus $x$ has a property that
\begin{equation}\label{eq:soc-soc}
(\rad A)x = 0 = x(\rad A).
\end{equation}
The nonzero map $f$ is regarded as a map either in
$\Hom_A(\Ds_{i\in I_0}P_{a_i}, \Ds_{j\in J_1}P_{b_j})$ in case (a); or
in $\Hom_A(\Ds_{j\in J_0}P_{b_j}, \Ds_{i\in I_1}P_{a_i})$ in case (b).
In case (a) we have a commutative diagram
$$
\xymatrix{
\Ds_{i\in I_1}P_{a_i} & \Ds_{i\in I_0}P_{a_i} & 0\\
0   &\Ds_{i\in J_1}P_{b_i} & \Ds_{i\in J_0}P_{b_i}
\ar"1,1";"1,2"^u
\ar"1,2";"1,3"
\ar"2,2";"2,3"_v
\ar"2,1";"2,2"
\ar"1,2";"2,2"^f
\ar"1,1";"2,1"
\ar"1,3";"2,3"
}
$$
by \eqref{eq:soc-soc} because $u, v$ are given by matrices all entries of which are left multiplications of elements of $\rad A$.

Consider the following diagram to compute
$\Hom_{\Kb(\proj A)}(P(w,s), P(w',s')[-1])$:
$$
\xymatrix{
\Ds_{i\in I_1}P_{a_i} & \Ds_{i\in I_0}P_{a_i} & 0\\
0   &\Ds_{i\in J_1}P_{b_i} & \Ds_{i\in J_0}P_{b_i}
\ar"1,1";"1,2"
\ar"1,2";"1,3"
\ar"2,2";"2,3"
\ar"2,1";"2,2"
\ar"1,2";"2,2"^f
\ar@{-->}"1,3";"2,2"
\ar@{-->}"1,2";"2,1"
}
$$
This shows that $\Hom_{\Kb(\proj A)}(P(w,s), P(w',s')[-1]) \ne 0$.
Thus $\Hom_{\Kb(\proj A)}(T, T[-1]) \ne 0$, and $T$ is not pretilting, a contradiction.

In case (b) the same argument applies to see that $T$ is not pretilting.

{\bf (2).}
Assume that $s(a) \ne s'(b)$.
Since $a \in w_1, b \in w'_1$, we have $a = a_g, b = b_h$ for some
$1 \le g \le m, 1\le h \le n$.
It follows from $s(a) \ne s'(b)$ that either (a) $g \in I_0, h \in J_1$
or (b) $g \in I_1, h \in J_0$.
Since $0 \ne \la_{p(a,b)} \in \Hom_A(P_{a_g}, P_{b_{h}})$ and
$0 \ne \la_{p(b,a)} \in \Hom_A(P_{b_h}, P_{a_{g}})$,
we can take a nonzero map $f:=\la_{p(a,b)}$ in
$\Hom_A(\Ds_{i\in I_0}P_{a_i}, \Ds_{j\in J_1}P_{b_j})$ in case (a); or
$f:=\la_{p(b,a)}$ in $\Hom_A(\Ds_{j\in J_0}P_{b_j}, \Ds_{i\in I_1}P_{a_i})$ in case (b).
Using the zero relations for Brauer tree algebras
(i.e., $\al_i\be^i=0, \be_i\al^i = 0$ for each vertex $i$ of $Q_G$)
we have $fu =0, vf=0$ in case (a), and $fv = 0, uf = 0$ in case (b).
The rest is the same as in the case (1).
Therefore in each case $T$ is not pretilting.
\end{proof}

%

The following proposition implies that $g$-polytope $\calP(A)$ is convex. 

\begin{prp}\label{union is convex}
We have
$$\calP(A) = 
\conv(g(M)\ |\ M \in \indtptilt A).$$ 

\end{prp}

\begin{proof}
By the definition of convex hull \eqref{eq:dfn-conv-hull}, 
we have $\conv_0(T) \subseteq \conv(g(M)\ |\ M \in \indtptilt A)$
for all $T \in \ttilt A$ and 
therefore $\calP(A)$ is included in the right hand side.
Now for each $M \in \indtptilt A$, there exists some $T \in \ttilt A$ such that $g(M)$ is contained in $\conv_0(T)$ \cite[Theorem 2.10]{AIR}.
Hence all $g(M)$ are contained in $\calP(A)$.
Therefore to show the converse inclusion it is enough to show that $\calP(A)$ is convex
by the definition of convex hulls
(Definition \ref{dfn:triang}(1)).

{To this end, it is enough to show the following by Proposition \ref{union will be convex}:}

\medskip
{$(*)$ $\ttilt A$ admits at most two indecomposable left approximation.}
\medskip
\\
Namely, for each $T\in\ttilt A$ with a decomposition
$T= \Ds_{i=1}^n T_i$ into indecomposable complexes $T_i \in\ttilt A$ ($1 \le i \le n$), and for any left mutation sequence  
$$\xymatrix{
T_i \ar[r]^(0.4){} & \Ds_{k \in I} T_k \ar[r] & T_i' \ar[r]&T_i[1]
}$$
with $T_i'\in\tptilt A$, we only have to show that $\#I = |\Ds_{k\in I} T_k |\leq 2$.

Now from the sequence above we have
\begin{equation}\label{eq:sum-g}
g(T_i)+g(T_i') = \sum_{k\in I} g(T_k).
\end{equation}
Furthermore, since $T_i, T_i', T_k \in\indtptilt A$\ ($k \in I$), it follows
by Proposition \ref{alter walk}
that $g(T_i), g(T_i')$ and $g(T_k)\ (k \in I)$ are given by some alternating signed walks $(w,s), (w', s')$ and $(w^{(k)}, s^{(k)})$,
respectively. 
First note that since $\Ds_{k \in I} T_k$ is pretilting,
we have
\begin{equation}\label{eq:walk-inside}
w_1^{(k)} \subseteq w_1 \cup w'_1 \ \text{ for all $k \in I$}
\end{equation}
by Lemma \ref{lem:tau-rigid-connected}(1).
Set $w'':= w \cap w'$.
We first consider the case that 
$w''_0 \ne \emptyset$, and therefore that $w''$ is a walk of a length $q \ge 0$.

{\bf Case 1.} $q > 0$.
Let $w'':= a_1a_2\cdots a_q$.
Then we may write
$$
\begin{aligned}
w &= b_1b_2\cdots b_p a_1a_2\cdots a_q c_1c_2\cdots c_r\\
w' &= b'_1b'_2\cdots b'_{p'} a_1a_2\cdots a_q c'_1c'_2\cdots c'_{r'}
\end{aligned}
$$
with $p,r,p',r' \ge 0$ as in the following figure.
$$
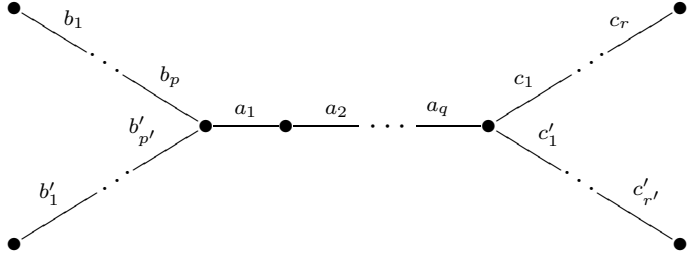
\vcenter{
\xymatrix@R=12pt@M=0pt{
\bullet&&&&&&&\bullet\\
&\ddots\ &&&&&\ \adots\\
&&\bullet&\bullet&\ \cdots\ &\bullet\\
&\ \adots&&&&&\ddots\ \\
\bullet&&&&&&&\bullet
\ar@{-}"1,1";"2,2"^{b_1}
\ar@{-}"2,2";"3,3"^{b_p}
\ar@{-}"3,3";"3,4"^{a_1}
\ar@{-}"3,4";"3,5"^{a_2}
\ar@{-}"3,5";"3,6"^{a_q}
\ar@{-}"3,6";"2,7"^{c_1}
\ar@{-}"2,7";"1,8"^{c_r}
\ar@{-}"4,2";"3,3"^{b'_{p'}}
\ar@{-}"3,6";"4,7"^{c'_1}
\ar@{-}"4,7";"5,8"^{c'_{r'}}
\ar@{-}"5,1";"4,2"^{b'_1}
}}
$$
Since $T_i \not\iso T'_i$, we have $(p,r, p', r') \ne (0,0,0,0)$.
Without loss of generality we may assume that $p > 0$.

We divide this case into two cases (a) and (b) below.

{\bf Case (a).} $s(a_1) = s'(a_1)$.
Take $k \in I$ such that $b_1 \in w^{(k)}_1$ and $s^{(k)}(b_1) = s(b_1)$.
Such a $k$ exists because
$s^{(j)}(b_1) \in \{0, 1, -1\}$ for all $j \in I$
and $\sum_{j\in I}s^{(j)}(b_{t-1}) = s(b_{t-1}) = \pm 1$.
(In this case, $k$ is unique by Lemma \ref{lem:tau-rigid-connected}(1).)
Then, since $\Ds_{k\in I} T_k$ is pretilting, it follows
by Lemma \ref{lem:tau-rigid-connected}(2) that
$\{b_1, b_2, \dots, b_p\} \subseteq w^{(k)}_1$.
Indeed, if $b_{t-1} \in w_1^{(k)}$
but $b_t \not\in w_1^{(k)}$ for some $2 \le t \le p$,
then there exists some $k' \in I$ such that
$b_t \in w_1^{(k')}$ and $s^{(k')}(b_{t}) = s(b_{t})$ by the same reason as above.
Here since
$s^{(k)}(b_{t-1}) = s(b_{t-1})= -s(b_{t})= -s^{(k')}(b_{t})$,
we have
$s^{(k)}(b_{t-1}) = - s^{(k')}(b_{t})$.
Note that $w_1^{(k)} \cap w_1^{(k')} = \emptyset$ by the formula \eqref{eq:walk-inside}
and by Lemma \ref{lem:tau-rigid-connected}(1).
Then
$P(w^{(k)},s^{(k)}) \ds P(w^{(k')},s^{(k')})$ cannot be
pretilting by Lemma \ref{lem:tau-rigid-connected}(2), a contradiction (see the figure below).
$$
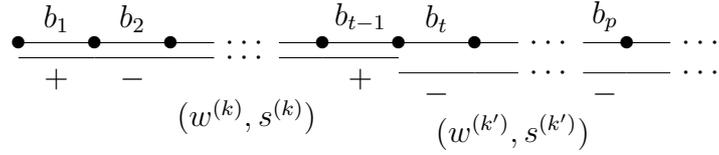
\begin{tikzpicture}
\node (A1) at (0,0) {$\bullet$};
\node (A2) at (1,0) {$\bullet$};
\node (A3) at (2,0) {$\bullet$};
\node (A4) at (3,0) {$\cdots$};
\node (A5) at (4,0) {$\bullet$};
\node (A6) at (5,0) {$\bullet$};
\node (A7) at (6,0) {$\bullet$};
\node (A8) at (7,0) {$\cdots$};
\node (A9) at (8,0) {$\bullet$};
\node (A10) at (9,0) {$\cdots$};
\path (A1.center) edge node[auto]{$b_1$} (A2.center)
(A2.center) edge node[auto]{$b_2$}(A3.center);
\draw (A3.center)--(A4)
--(A5.center);
\path (A5.center) edge node [auto] {$b_{t-1}$} (A6.center)
(A6.center) edge node[auto]{$b_t$} (A7.center);
\draw (A7.center)--(A8);
\path (A8) edge node[auto]{$b_p$}
(A9.center)
(A9.center) edge (A10);
\node (B1) at (0,-0.2) {};
\node (B2) at (1,-0.2) {};
\node (B3) at (2,-0.2) {};
\node (B4) at (3,-0.2) {$\cdots$};
\node (B5) at (4,-0.2) {};
\node (B6) at (5,-0.2) {};
\path  (B1.center) edge node[below]{$+$} (B2.center)
(B2.center) edge node[below]{$-$}(B3.center);
\draw  (B3.center)--(B4)
--(B5.center);
\path (B5.center) edge node [below] {$+$} (B6.center) node[below=4mm, 
xshift=-10mm, black]{$(w^{(k)},s^{(k)})$};
\node (C6) at (5,-0.4) {};
\node (C7) at (6,-0.4) {};
\node (C8) at (7,-0.4) {$\cdots$};
\node (C9) at (8,-0.4) {};
\node (C10) at (9,-0.4) {$\cdots$};
\path (C6.center) edge node[below]{$-$} (C7.center)node [below=4mm, xshift=15mm]{$(w^{(k')},s^{(k')})$};
\draw (C7.center)--(C8);
\path (C8) edge node[below]{$-$}
(C9.center)
(C9.center) edge (C10);
\end{tikzpicture}
$$
Similarly we also have that
$\{b_1, b_2, \dots, b_p, \linebreak[2]
a_1, a_2,\dots, a_q \} \subseteq w^{(k)}_1$.
Then also $c_1$ or $c'_1$ is contained in $w^{(k)}_1$.
In the former case, we have $w_1 \subseteq w^{(k)}_1$ and in the latter case, we have 
$$
\{b_1, b_2, \dots, b_p, a_1, a_2,\dots, a_q,c_1'c_2'\dots c_{r'}' \} \subseteq w^{(k)}_1.
$$
By the formula \eqref{eq:walk-inside} we have
$$
w^{(k)} = w \quad\text{ or }\quad
w^{(k)} = b_1b_2\cdots b_pa_1a_2\cdots a_qc_1'c_2'\cdots c_{r'}'.
$$
But, if $w = w^{(k)}$,
then it follows from $s(b_1) = s^{(k)}(b_1)$ that also $s = s^{(k)}$,
and hence $T_i \iso T_k$, a contradiction.
Thus, we have
$w^{(k)} = b_1b_2\cdots b_pa_1a_2\cdots a_qc_1'c_2'\cdots c_{r'}'$.

Now since
$\sum_{j\in I}s^{(j)}(a_1) = s(a_1) + s'(a_1) = \pm 2$,
there exists one more $\ell \in I$ such that $a_1 \in w^{(\ell)}_1$ and
$s(a_1) = s^{(\ell)}(a_1)$.
Then by a similar argument we have
$$
w^{(\ell)} = b'_1b'_2\cdots b'_{p'}a_1a_2\cdots a_q c_1c_2\cdots c_{r}.
$$
Hence by Lemma \ref{lem:tau-rigid-connected}(1)
we see that $I = \{k, \ell\}$.
As a consequence, $\#I = 2$
as in the following figure.
$$
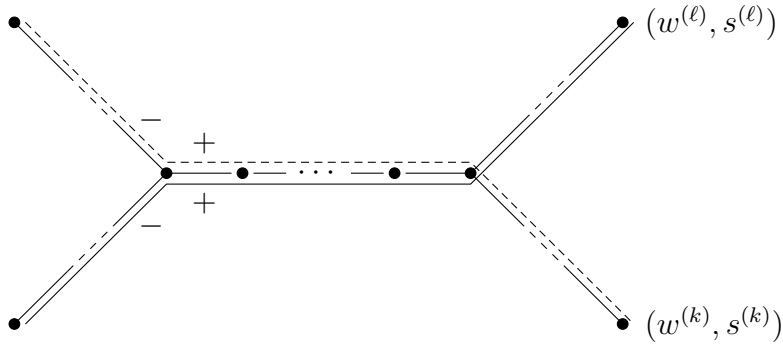
\begin{tikzpicture}
\node (A1) at (-4,2) {};
\node (A2) at (-3.3,1.3) {};
\node (A3) at (-2.7,0.7) {};
\node (B1) at (-2,0) {};
\node (B2) at (-1,0) {};
\node (B3) at (0,0) {$\cdots$};
\node (B4) at (1,0) {};
\node (B5) at (2,0) {};
\node (C1) at (4,2) {};
\node (C2) at (3.3,1.3) {};
\node (C3) at (2.7,0.7) {};
\node (D1) at (-4,-2) {};
\node (D2) at (-3.3,-1.3) {};
\node (D3) at (-2.7,-0.7) {};
\node (E1) at (4,-2) {};
\node (E2) at (3.3,-1.3) {};
\node (E3) at (2.7,-0.7) {};
\foreach \point in {A1,B1,B2,B4,B5,C1,D1,E1}
\draw [fill] (\point) circle (2pt);
\draw (A1.center) -- (A2.center);
\draw (A2.center) [dashed]-- (A3.center);
\draw (A3.center)-- (B1.center) -- (B2) --(B3) -- (B4) -- (B5.center) --(C3.center);
\draw (C3.center) [dashed] -- (C2.center);
\draw (C2.center) -- (C1.center);
\draw (D1.center) -- (D2.center);
\draw (D2.center) [dashed] -- (D3.center);
\draw (D3.center) -- (B1.center);
\draw (E1.center) -- (E2.center);
\draw (E2.center) [dashed] -- (E3.center);
\draw (E3.center) -- (B5.center);
\draw [densely dashed] (A1.east) -- (B1.north) -- (B5.north) -- (E1.east) node[black, right]{$(w^{(k)},s^{(k)})$};
\draw  (D1.east) -- (B1.south) -- (B5.south) -- (C1.east)node[black, right]{$(w^{(\ell)},s^{(\ell)})$};
\node at (-1.5,0.4) {$+$};
\node at (-1.5,-0.4) {$+$};
\node at (-2.2,0.7) {$-$};
\node at (-2.2,-0.7) {$-$};
\end{tikzpicture}
$$

{\bf Case (b).} $s(a_1) = -s'(a_1)$.
In this case
\begin{equation}\label{eq:b-b'}
\begin{aligned}
s(b_p) &= -s'(b'_{p'}), \text{ and}\\
\sum_{j\in I}s^{(j)}(a_t) &= s(a_t) + s'(a_t) = 0\ 
\text{for all }t \in \{1, \dots, q\}.
\end{aligned}
\end{equation}
There exists some $w^{(k)}$ such that
$b_1 \in w^{(k)}_1$ and $s^{(k)}(b_1) = s(b_1)$.
By the same argument as in the case (a) using the equality \eqref{eq:b-b'}
we see that
$\{b_1, \dots, b_p, b'_{p'}, \dots, b'_1\} \subseteq w^{(k)}_1$.
If $r = r' = 0$, then $w^{(k)}_1 = b_1\dots, b_p b'_{p'} \dots, b'_1$
and we have $\#I = 1$
by the same argument as in (a).

Otherwise we may assume $r \ge 1$, and similarly there exists some 
$w^{(\ell)}$ such that
$c_r \in w^{(\ell)}_1$ and $s^{(\ell)}(c_r) = s(c_r)$ and we see that
$w^{(\ell)} = c_r\dots c_1 c'_{1} \dots c'_{r'}$.
In this case $\#I =2$ as in the following figure.
$$
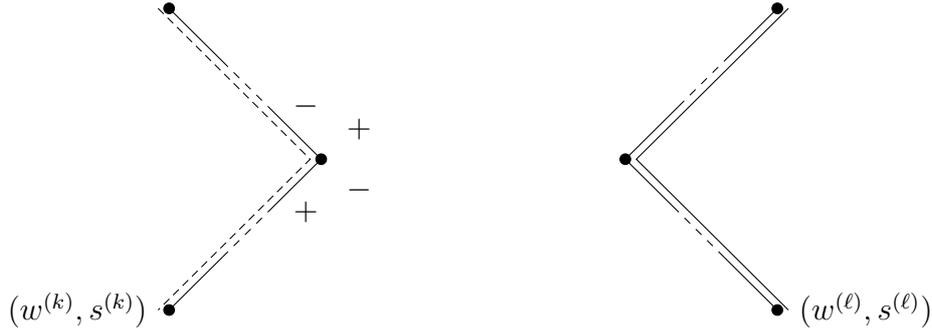
\begin{tikzpicture}
\node (A1) at (-4,2) {};
\node (A2) at (-3.3,1.3) {};
\node (A3) at (-2.7,0.7) {};
\node (B1) at (-2,0) {};
\node (B2) at (-1,0) {};
\node (B3) at (0,0) {};
\node (B4) at (1,0) {};
\node (B5) at (2,0) {};
\node (C1) at (4,2) {};
\node (C2) at (3.3,1.3) {};
\node (C3) at (2.7,0.7) {};
\node (D1) at (-4,-2) {};
\node (D2) at (-3.3,-1.3) {};
\node (D3) at (-2.7,-0.7) {};
\node (E1) at (4,-2) {};
\node (E2) at (3.3,-1.3) {};
\node (E3) at (2.7,-0.7) {};
\foreach \point in {A1,B1,B5,C1,D1,E1}
\draw [fill] (\point) circle (2pt);
\draw (A1.center) -- (A2.center);
\draw (A2.center) [dashed]-- (A3.center);
\draw (A3.center)-- (B1.center);
\draw (B5.center) -- (C3.center);
\draw (C3.center) [dashed] -- (C2.center);
\draw (C2.center) -- (C1.center);
\draw (D1.center) -- (D2.center);
\draw (D2.center) [dashed] -- (D3.center);
\draw (D3.center) -- (B1.center);
\draw (E1.center) -- (E2.center);
\draw (E2.center) [dashed] -- (E3.center);
\draw (E3.center) -- (B5.center);
\draw [densely dashed] (A1.west) -- (B1.west) -- (D1.west) node[black, left]{$(w^{(k)},s^{(k)})$};
\draw  (C1.east) -- (B5.east) -- (E1.east)node[black, right]{$(w^{(\ell)},s^{(\ell)})$};
\node at (-1.5,0.4) {$+$};
\node at (-1.5,-0.4) {$-$};
\node at (-2.2,0.7) {$-$};
\node at (-2.2,-0.7) {$+$};
\end{tikzpicture}
$$

{\bf Case 2.} $q = 0$.  This is an easier case, and
a similar argument works to show that $\#I \le 2$.

Finally consider the remaining case where $w_0'' = \emptyset$.
Then by the same argument as above we have $(w,s) = (w^{(k)}, s^{(k)})$ for some $k \in I$,
which implies that $i \in I$, a contradiction.
\end{proof}


\begin{rmk}
{The proof of $(*)$ above gives a description of $\Ds_{k \in I} T_k$.
There is an alternative proof without this description as follows,
which was informed us by the referee
and by Toshitaka Aoki.}

{
{\bf Case 1.} Assume that $T=A=\Ds_{i=1}^n e_iA$. 
In this case, 
we have left mutation sequence  
$$\xymatrix{
e_iA \ar[r]^(0.4){} &  \Ds_{k \in Q_0} e_kA \ar[r] & e_iA' \ar[r]&e_iA[1]
}$$
where $k$ runs over all vertices which are the target of outgoing arrows of $Q$ from $i$.
Then, since $A$ is a Brauer tree algebra, we have $|\Ds_{k \in Q_0} e_k A|\leq 2$ .}

{
{\bf Case 2.} 
Next we consider the general case. 
Since $T$ is a tilting complex, it gives an equivalence functor 
$F_T: \Kb(\proj A)\to \Kb(\proj A')$, where $A':=\End_{\Kb(\proj A)}(T)$. 
In particular, $F_T$ sends the above mutation sequence to the one of $\Kb(\proj A')$ and the number of indecomposable direct summands is preserved. 
Thus,  $|\Ds_{k\in I} T_k |\leq 2$ follows from Case 1 since $A'$ is also a Brauer tree algebra.}

\end{rmk}

As a consequence, we have the following result, which shows that 
the polytope $\calP(A)$ has a symmetric shape. 

\begin{cor}\label{same convex}
For each $i \in Q_0$,  we have
$$\bigcup_{T \in \tiny\ttilt A_i^\leq} \conv_0(T) = 
\conv(g(M)\ |\ M \in \indtptilt A_i^\leq),$$ 
$$\bigcup_{T \in \tiny\ttilt A_i^\geq} \conv_0(T) = \conv(g(M)\ |\ M \in \indtptilt A_i^\geq).$$
Moreover, we have 
$$\bigcup_{T \in \tiny\ttilt A_i^\geq} \conv_0(T) =-\bigcup_{T \in \tiny\ttilt A_i^\leq} \conv_0(T).
$$
\end{cor}

\begin{proof}
We show the first equality. The second one is shown similarly.

By Lemma \ref{upper or lower} (3), 
the left hand side is $\calP(A)\cap H_i^\leq$. 
On the other hand, Proposition \ref{union is convex} implies that 
$\calP(A) = \conv(g(M)\ |\ M \in \indtptilt A).$ 
Thus the equality follows from Lemma \ref{upper or lower} (1) and (2). 
Moreover, Lemma \ref{divide half} implies that 
$$
\begin{array}{rcl}
\conv(g(M)\ |\ M \in \indtptilt A_i^\geq) & = &\conv(-g(M^\vee)\ |\ M^\vee \in \indtptilt A_i^\leq)\\
& = &-\conv(g(M^\vee)\ |\ M^\vee \in \indtptilt A_i^\leq).
\end{array}$$
Therefore we get the third equality.
\end{proof}

Now we give a proof of Theorem \ref{thm:main1}.

\begin{proof}[Proof of Theorem \ref{thm:main1}]
Let $\calP:=\bigcup_{T \in \tiny\ttilt A_i^\leq} \conv_0(T)$ and $\calQ:=\bigcup_{T \in \tiny\ttilt A_i^\geq} \conv_0(T)$. 
Since we have $\calP=-\calQ$ by Corollary \ref{same convex}, 
the $f$-vectors of the simplicial complexes induced by their unimodular triangulations coincide by Theorem \ref{thm:BM}.
Then Lemma \ref{upper or lower} (4) proves the {assertion}. 
\end{proof}


\section{Mutation and 2-term pretilting complexes}
In this section, we show that the $f$-vector, or equivalently, 
the number of $\tpjtilt(A_G)$ depends only on the number of $G$ and it is independent of the shape of $G$. For this purpose, we study 
the relation between simplicial complexes associated to algebras
which are derived equivalent to each other.  
In particular, we show that the derived equivalence functor given by mutation induces 
a correspondence between one half of one of the simplicial complexes to one half of the other. 

The aim of this section is to show the following result.

\begin{thm} \label{thm:main2}
Let $G$ be a Brauer tree with an arbitrary multiplicity and 
$A_G$ the Brauer tree algebra of $G$. 
Then the $f$-vector of $\Delta(A_G)$ only depends on the number of edges of $G$. In particular, it is invariant under derived equivalences.
\end{thm}

In the sequel, let $G$ be a Brauer tree with multiplicity 1 and $A_G$ the Brauer tree algebra of $G$. We keep Notation \ref{notation}, and we set $A:= A_G$ and $n:= |G|=|A|$.


We fix $i\in Q_0$.
Then we define the $2$-term tilting complex obtained by the left mutation from $A$
$$
\mu_i^L (A):= \mu_{e_i A}^L (A) :=
\xymatrix@R=10pt{(\cdots
0 \ar[r] & e_i A \ar[r]^-{\left(\smat{f\\0}\right)}  & \left(\Ds_{\smat{j\in Q_0\\i\to j}}e_j A\right) \ds (1-e_i) A \ar[r] & 0\cdots),  \\
}
$$
where $f$ is the left minimal $\add ((1-e_i) A)$-approximation \cite{AI}.
One can define the right mutation $\mu_i^R$ dually.

The following  results follow from \cite{K,R1} (see also  \cite{S2}).

\begin{prp} \label{Rmk:Kmove}
\begin{enumerate}
\item 
For any $i\in Q_{0}$, we have $\End_{\Kb(\proj A)} (\mu_i^L (A_G) ) \iso A_{\mu_i (G)}$, where $\mu_i (G)$ denotes the Kauer move obtained from $G$.  
\item Let $G^{\prime} $ be a Brauer tree. 
Then
$ |G| = |G^{\prime}| $ if and only if there is a sequence 
$G:=G^0$, $G^{k+1}:=\mu_{i_{k}} (G^{k})$ and $G^{\prime}:=G^m$
of Brauer trees and $i_k\in Q_{{G^k},0}$ such that 
$\End_{\Kb(\proj A)} (\mu_{i_k}^L (A_{G^k}) ) \iso A_{G^{k+1}}$ for any $0 \leq k \leq m-1$.
In particular $A_G$ and $A_{G^{\prime}}$ are derived equivalent. 
\end{enumerate}
\end{prp}

For the convenience of the reader,  
we briefly recall the notion of Kauer move, using the following example.

\begin{exm}For a given Brauer tree graph $G$, we get the new Brauer graph $\mu_i(G)$ by applying the Kauer move associated to the edge $i$ as follows:
$$(1) \ \ 
G=\vcenter
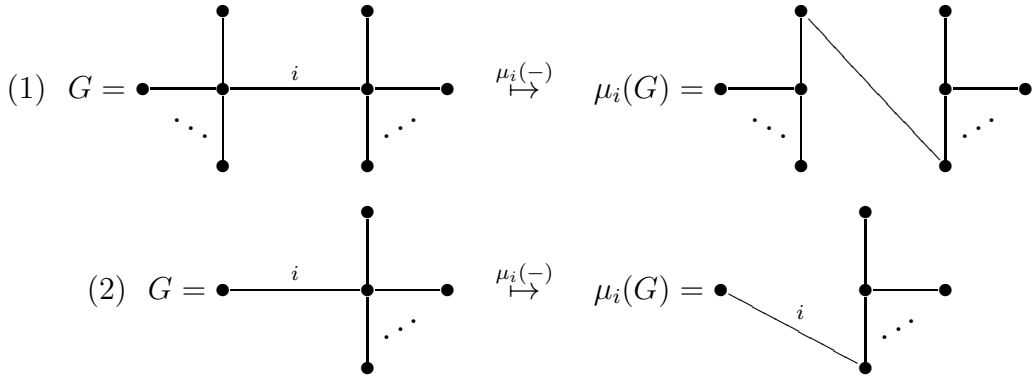
{
\xymatrix@M=0pt{
& \bullet && \bullet\\
\bullet &\bullet&&\bullet&\bullet\\
& \bullet && \bullet&
\ar@{-}"1,2";"2,2"
\ar@{-}"2,1";"2,2"
\ar@{-}"1,2";"2,2"
\ar@{-}"3,2";"2,2"
\ar@{-}"1,4";"2,4"
\ar@{-}"3,4";"2,4"
\ar@{-}"2,5";"2,4"
\ar@{-}"2,2";"2,4"^i
\ar@{}"3,1";"2,2"|(0.6){\Large\ddots}
\ar@{}"3,5";"2,4"|(0.6){\Large\adots}
}}
\quad
\overset{\mu_i(-)}{\mapsto}
\quad
\mu_i(G)=\vcenter{
\xymatrix@M=0pt{
& \bullet && \bullet\\
\bullet &\bullet&&\bullet&\bullet\\
& \bullet && \bullet&
\ar@{-}"1,2";"2,2"
\ar@{-}"2,1";"2,2"
\ar@{-}"1,2";"2,2"
\ar@{-}"3,2";"2,2"
\ar@{-}"1,4";"2,4"
\ar@{-}"3,4";"2,4"
\ar@{-}"2,5";"2,4"
\ar@{-}"1,2";"3,4"
\ar@{}"3,1";"2,2"|(0.6){\Large\ddots}
\ar@{}"3,5";"2,4"|(0.6){\Large\adots}
}}
$$

$$(2) \ \ 
G=\vcenter{
\xymatrix@M=0pt{
  &&\bullet \\
 \bullet&&\bullet&\bullet\\
  && \bullet&
\ar@{-}"1,3";"2,3"
\ar@{-}"3,3";"2,3"
\ar@{-}"2,4";"2,3"
\ar@{-}"2,1";"2,3"^i
\ar@{}"3,4";"2,3"|(0.6){\Large\adots}
}}
\quad
\overset{\mu_i(-)}{\mapsto}
\quad
\mu_i(G)=\vcenter{
\xymatrix@M=0pt{
  &&\bullet \\
 \bullet&&\bullet&\bullet\\
  && \bullet&
\ar@{-}"1,3";"2,3"
\ar@{-}"3,3";"2,3"
\ar@{-}"2,4";"2,3"
\ar@{-}"2,1";"3,3"^i
\ar@{}"3,4";"2,3"|(0.6){\Large\adots}
}}
$$
\end{exm}

Next, we introduce the following terminology. 

\begin{dfn}
Let $X, Y\in \tptilt(A)$. 
Then denote by $X \geq Y$ if $\Hom (X, Y[i])=0$ for all positive integer $i >0$.
Since complexes are 2-term, this condition is equivalent to the condition that $\Hom (X, Y[1])=0$. 
Note that $\geq$ is a partial order of $\ttilt(A)$ \cite{AI}.  
\end{dfn}

We fix $i \in Q_0$ and an integer $j$ such that $1\leq j\leq n$.

\begin{lem} \label{Lem:cal}The following equalities hold.
\begin{enumerate}
\item $\tpjtilt(A)_i^\leq = \{ T \in \tpjtilt(A) \mid \mu_i^L (A) \geq T \}$.
\item $\tpjtilt(A)_i^\geq =\{ T \in \tpjtilt(A) \mid T \geq \mu_i^R (A[1]) \} $.
\end{enumerate}
\end{lem}
\begin{proof}
(1) 
Let $X= (\cdots 0 \to X^{-1} \to X^{0} \to 0 \cdots )\in \tpjtilt(A)_i^\leq$.
We will show that $\Hom_{\Kb(\proj A)} (\mu_i^L (A), X[1])=0$.
Since $e_i A \not\in \add X^{0}$ and $f$ is the approximation, there exists a homomorphism $f^{\prime}$ such that the following diagram commutes:
$$
\xymatrix{
 & e_i A \ar[r]^(0.25){\left(\smat{f\\0}\right)} \ar[d] & \left(\Ds_{\smat{j\in Q_0\\i\to j}}e_j A\right) \ds (1-e_i)A \ar@{-->}[dl]^{f^{\prime}} \\
X^{-1} \ar[r] & X^0 & 
}
$$

Thus we have $\Hom_{\Kb(\proj A)} (\mu_i^L (A), X[1])=0$.

Conversely, let $X\in \{ T \in \tpjtilt(A) \mid \mu_i^L (A) \geq T \}$.
If there exists $e_i A \in \add X^0$, 
then we can take the canonical inclusion $\iota_i$ from $e_i A$ to $X^0$. On the other hand, our assumption implies that  
there exist homomorphism $a,b$ such that $\iota_i =ga+ b_1 f$:
$$
\xymatrix{
 & e_i A \ar[r]^(0.25){\left(\smat{f\\0}\right)} \ar[d]_{\iota_i} \ar@{-->}[dl]_{a} & \left(\Ds_{\smat{j\in Q_0\\i\to j}}e_j A\right) \ds (1-e_i)A \ar@{-->}[dl]^{b=(\smat{b_1 & b_2})} \\
X^{-1} \ar[r]^g & X^0 & 
}
$$
Hence, we have 
$$\id_{e_i A}=\pi_i \iota_i \in \rad_A (e_i A ,e_i A),$$
where $\pi_i$ is the canonical projection from $X$ to $e_i A$,
which is a contradiction.

(2) 
The second equality holds dually.
\end{proof}

Next we give the following lemma.

\begin{lem} \label{Lem:num}
\begin{enumerate}
\item We have $\# \tpjtilt(A_G)_i^\leq = \#\tpjtilt(A_{\mu_i(G)})_i^\geq$.
\item We have $\# \tpjtilt(A_G)_i^0 = \#\tpjtilt(A_{\mu_i(G)})_i^0$.
\end{enumerate}
\end{lem}

\begin{proof}
(1) By Proposition \ref{Rmk:Kmove}, we have $\End_{\Kb(\proj A)} (\mu_i^L (A_G) ) \iso A_{\mu_i (G)}$. 
Then there exists the induced equivalence
$F:\Kb(\proj A_G)\to\Kb(\proj A_{\mu_i(G)})$. It sends $\mu_i^L (A_G)$ to $A_{\mu_i(G)}$, and 
 $e_k(A_G)$ to $e'_k A_{\mu_i (G)}$ for any $k\neq i\in Q_{G,0}$, where
$e'_k$ is the idempotent $\mu^L_i(A_G) \twoheadrightarrow e_kA_G \incl \mu^L_i(A_G)$ of $A_{\mu_i(G)}$.
Note that $F$ preserves $\geq$ and mutations because $F$ is an equivalence.
Thus, we have $ F(A_G [1]) \iso  \mu_i^R (A_{\mu_i (G)} [1])$.
Therefore, by Lemma~\ref{Lem:cal}, we have the equalities:
$$
\begin{array}{rcl}
\# \tpjtilt(A_G)_i^\leq  & = & \# \{ T \in \tpjtilt(A_G) \mid \mu_i^L (A_G) \geq T \} \\
     & = & \# \{ T \in \tpjtilt(A_G) \mid \mu_i^L (A_G) \geq T \geq A_G [1] \} \\
     & = & \# \{ T \in \tpjtilt(A_{\mu_i (G)}) \mid A_{\mu_i (G)} \geq T \geq \mu_i^R (A_{\mu_i (G)} [1]) \} \\
     & = & \# \{ T \in \tpjtilt(A_{\mu_i (G)}) \mid  T \geq \mu_i^R (A_{\mu_i (G)} [1]) \}\\
     & = & \# \tpjtilt(A_G)_i^\geq. 
\end{array}
$$
Thus we get (1). 
Since $F$ gives an equivalence 
$\add  (\mu_i^L (A_G))  \to \proj A_{\mu_i(G)}$, we have (2).
\end{proof}

\begin{rmk}\label{rem poset}
By Lemma~\ref{Lem:num}, we have a poset isomorphism from $\ttilt(A_G)_i^\leq$ to $\ttilt(A_{\mu_i(G)})_i^\geq$ as shown in the
following example.
\end{rmk}

\begin{exm}
Consider the following Brauer trees.
$$
G:=
\vcenter{
\xymatrix@M=0pt{
& \bullet\\
\bullet& \bullet &\bullet
\ar@{-}"1,2";"2,2"^1
\ar@{-}"2,2";"2,3"_3
\ar@{-}"2,1";"2,2"_2
}},
\quad
H:= \mu_1(G)=
\vcenter{
\xymatrix@M=0pt{
\bullet&\bullet& \bullet &\bullet
\ar@{-}"1,1";"1,2"^1
\ar@{-}"1,2";"1,3"^2
\ar@{-}"1,3";"1,4"^3
}}.
$$ 
Then $\ttilt(A_G)_1^\leq$ and $\ttilt(A_H)_1^\geq$ are given by the thick arrow parts of the left and the right figure,
respectively.
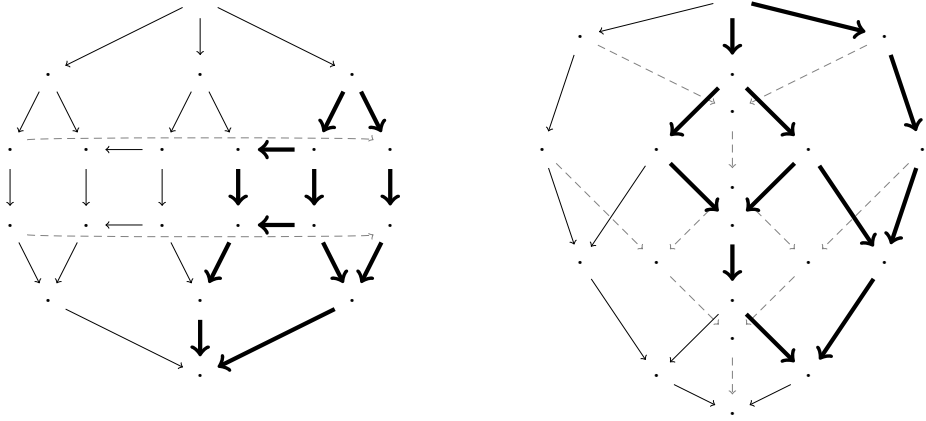
\[
\begin{tikzpicture}[every node/.style={circle,minimum size=0.05,fill=white}]
\draw (0,0) node (a) {.};
\draw (0,-1) node (b) {.};
\draw (0,-4) node (c) {.};
\draw (0,-5) node (d) {.};
\draw (0.5,-2) node (e) {.};
\draw (0.5,-3) node (f) {.};
\draw (1.5,-2) node (g) {.};
\draw (1.5,-3) node (h) {.};
\draw (2,-1) node (i) {.};
\draw (2,-4) node (j) {.};
\draw (2.5,-2) node (k) {.};
\draw (2.5,-3) node (l) {.};
\draw (-0.5,-2) node (-e) {.};
\draw (-0.5,-3) node (-f) {.};
\draw (-1.5,-2) node (-g) {.};
\draw (-1.5,-3) node (-h) {.};
\draw (-2,-1) node (-i) {.};
\draw (-2,-4) node (-j) {.};
\draw (-2.5,-2) node (-k) {.};
\draw (-2.5,-3) node (-l) {.};
\draw[draw=gray,->,ultra thin,densely dashed,bend left,distance=2] (-k) to (k);
\draw[draw=gray,->,ultra thin,densely dashed,bend right,distance=2] (-l) to (l);
\draw[draw=black,->,ultra thin] (a)--(b);
\draw[draw=black,->,ultra thin] (b)--(e);
\draw[draw=black,->,ultra thick] (e)--(f);
\draw[draw=black,->,ultra thick] (f)--(c);
\draw[draw=black,->,ultra thick] (c)--(d);
\draw[draw=black,->,ultra thin] (a)--(i);
\draw[draw=black,->,ultra thick] (i)--(g);
\draw[draw=black,->,ultra thick] (g)--(h);
\draw[draw=black,->,ultra thick] (h)--(j);
\draw[draw=black,->,ultra thick] (j)--(d);
\draw[draw=black,->,ultra thick] (i)--(k);
\draw[draw=black,->,ultra thick] (k)--(l);
\draw[draw=black,->,ultra thick] (l)--(j);
\draw[draw=black,->,ultra thick] (g)--(e);
\draw[draw=black,->,ultra thick] (h)--(f);
\draw[draw=black,->,ultra thin] (b)--(-e);
\draw[draw=black,->,ultra thin] (-e)--(-f);
\draw[draw=black,->,ultra thin] (-f)--(c);
\draw[draw=black,->,ultra thin] (a)--(-i);
\draw[draw=black,->,ultra thin] (-i)--(-g);
\draw[draw=black,->,ultra thin] (-g)--(-h);
\draw[draw=black,->,ultra thin] (-h)--(-j);
\draw[draw=black,->,ultra thin] (-j)--(d);
\draw[draw=black,->,ultra thin] (-i)--(-k);
\draw[draw=black,->,ultra thin] (-k)--(-l);
\draw[draw=black,->,ultra thin] (-l)--(-j);
\draw[draw=black,->,ultra thin] (-e)--(-g);
\draw[draw=black,->,ultra thin] (-f)--(-h);
\draw (7,0) node (A) {.};
\draw (7,-1) node (B) {.};
\draw (7,-1.5) node (C) {.};
\draw (7,-2.5) node (D) {.};
\draw (7,-3) node (E) {.};
\draw (7,-4) node (F) {.};
\draw (7,-4.5) node (G) {.};
\draw (7,-5.5) node (H) {.};
\draw (8,-2) node (I) {.};
\draw (8,-3.5) node (J) {.};
\draw (8,-5) node (K) {.};
\draw (9,-0.5) node (L) {.};
\draw (9,-3.5) node (M) {.};
\draw (9.5,-2) node (N) {.};
\draw (6,-2) node (-I) {.};
\draw (6,-3.5) node (-J) {.};
\draw (6,-5) node (-K) {.};
\draw (5,-0.5) node (-L) {.};
\draw (5,-3.5) node (-M) {.};
\draw (4.5,-2) node (-N) {.};
\draw[draw=gray,->,ultra thin,densely dashed] (L) to (C);
\draw[draw=gray,->,ultra thin,densely dashed] (C) to (D);
\draw[draw=gray,->,ultra thin,densely dashed] (D) to (J);
\draw[draw=gray,->,ultra thin,densely dashed] (J) to (G);
\draw[draw=gray,->,ultra thin,densely dashed] (G) to (H);
\draw[draw=gray,->,ultra thin,densely dashed] (N) to (J);
\draw[draw=gray,->,ultra thin,densely dashed] (-L) to (C);
\draw[draw=gray,->,ultra thin,densely dashed] (D) to (-J);
\draw[draw=gray,->,ultra thin,densely dashed] (-J) to (G);
\draw[draw=gray,->,ultra thin,densely dashed] (-N) to (-J);
\draw[draw=black,->,ultra thick] (A) to (B);
\draw[draw=black,->,ultra thick] (B) to (I);
\draw[draw=black,->,ultra thick] (I) to (E);
\draw[draw=black,->,ultra thick] (E) to (F);
\draw[draw=black,->,ultra thick] (F) to (K);
\draw[draw=black,->,ultra thin] (K) to (H);
\draw[draw=black,->,ultra thick] (A) to (L);
\draw[draw=black,->,ultra thick] (L) to (N);
\draw[draw=black,->,ultra thick] (N) to (M);
\draw[draw=black,->,ultra thick] (M) to (K);
\draw[draw=black,->,ultra thick] (I) to (M);
\draw[draw=black,->,ultra thick] (B) to (-I);
\draw[draw=black,->,ultra thick] (-I) to (E);
\draw[draw=black,->,ultra thin] (F) to (-K);
\draw[draw=black,->,ultra thin] (-K) to (H);
\draw[draw=black,->,ultra thin] (A) to (-L);
\draw[draw=black,->,ultra thin] (-L) to (-N);
\draw[draw=black,->,ultra thin] (-N) to (-M);
\draw[draw=black,->,ultra thin] (-M) to (-K);
\draw[draw=black,->,ultra thin] (-I) to (-M);
\end{tikzpicture}
\]

\end{exm}

Now we are ready to show Theorem \ref{thm:main2}.

\begin{proof}[Proof of Theorem \ref{thm:main2}]\label{proof4.1}
First, consider the case that $G$ has the multiplicity 1. 
Let $i\in Q_0$. 
By Theorem \ref{thm:main1} and Lemma~\ref{Lem:num}, 
we have the following equalities: 
$$
\begin{array}{rcl}
\# \tpjtilt{A_G} & = & \# \tpjtilt(A_G)_i^\leq\ +\ \# \tpjtilt(A_G)_i^\geq\ -\ \# \tpjtilt(A_G)_i^0\\
 & = & 2 \# \tpjtilt(A_G)_i^\leq\  -\ \# \tpjtilt(A_G)_i^0\\
 & = & 2 \# \tpjtilt(A_{\mu_i(G)})_i^\geq\ -\ \# \tpjtilt(A_{\mu_i(G)})_i^0\\
 & = & \# \tpjtilt(A_{\mu_i(G)})_i^\leq\ +\ \# \tpjtilt(A_{\mu_i(G)})_i^\geq\ -\ \# \tpjtilt(A_{\mu_i(G)})_i^0 \\
 & = & \# \tpjtilt(A_{\mu_i(G)}). \\
\end{array}
$$
By Proposition \ref{Rmk:Kmove}, if $|G|=|G'|$, then 
$G$ and $G'$ can be related by applying Kauer moves repeatedly. 
Thus we have proved Theorem \ref{thm:main2}
for Brauer tree algebras with multiplicity 1. 
Moreover, we can explain the same consequence for the case of an arbitrary multiplicity as follows. 
Proposition \ref{alter walk} holds for Brauer tree algebras with an arbitrary multiplicity by \cite[Theorem 4.6]{AAC} (this is just because the isoclasses of indecomposable 2-term pretilting complexes does not depend on the multiplicity). 
Thus the above arguments work similarly and we get the same conclusion  (this fact also follows from the result by \cite{EJR}.)
\end{proof}

\section{Enumeration of $f$-vectors}
In this section, we give {explicit descriptions of the $f$-polynomial and the $h$-polynomial} of $\Delta(A_G)$
for Brauer tree algebras $A_G$.
As an application of our results, 
we will give the formulae for biCatalan numbers in the sense of \cite{BR}. 
We denote by $\binom{n+j}{j,j,n-j}:=(n+j)!/j! j! (n-j)!$ for $0\leq j\leq n$.
We have the following result. 

\begin{thm}\label{main3}
Let $G$ be a Brauer tree and $A_G$ the Brauer tree algebra of $G$.  Then
the $f$-polynomial and the $h$-polynomial of $\Delta(A_G)$ are 
given as follows $:$ 
$$f(x) = \sum_{j=0}^{n}\binom{n+j}{j,j,n-j}x^{n-j},\ \ \ 
h(x) = \sum_{k=0}^{n}\binom{n}{j}^2x^{n-j}.$$
In particular, the number of 2-term tilting complexes of $\Kb(\proj  A_G)$ is $\binom{2n}{n}$.
\end{thm}

\begin{proof}
Let $G$ be the following star-shaped Brauer tree :
$$
G=\vcenter
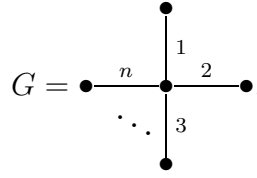
{
\xymatrix@M=0pt{
& \bullet  \\
\bullet &\bullet&\bullet\\\
& \bullet &
\ar@{-}"1,2";"2,2"^1
\ar@{-}"2,1";"2,2"^n
\ar@{-}"1,2";"2,2"
\ar@{-}"3,2";"2,2"_3
\ar@{-}"2,2";"2,3"^2
\ar@{}"3,1";"2,2"|(0.6){\Large\ddots}
}}.$$
By Theorem \ref{thm:main2}, it is enough to show that the statement holds for $\Delta(A_G)$. 
It is easy to check that 
$$\indtptilt A_G = \{P_i,P_i[1], (P_i \to P_j)\ |\ 1\leq i,j\leq n,i\neq j  \},$$
where $(P_i\to P_j)$ denotes the indecomposable 2-term 
pretilting complex. 

On the other hand, we let 
$$\calP_{\bbA_{n}} :=\conv ({e}_i-{e}_j\ |\ 1\leq i,j\leq n+1,i\neq j),$$
which is called the \emph{root lattice polytope of type $\bbA_{n}$} \cite{ABHPS}. 
It is an $n$-dimensional polytope in $\mathbb{R}^{n+1}$ contained in the hyperplane $\{v\in \mathbb{R}^{n+1}\ |\ \sum_{k=1}^{n+1} v_k=0\}$.

Then by taking $[P_i]:={e}_1-{e}_{i+1}$ for any $i$, we have 
$\calP(A_G)=\conv(g(M)\ |\ M\in\indtptilt A_G)=\conv([P_i],-[P_i], -[P_i] + [P_j]\ |\ 1\leq i,j\leq n  )\simeq\calP_{\bbA_{n}}.$ 
Because $\Delta(A_G)$ is given by one of the unimodular triangulations of  $\calP_{\bbA_{n}}$, the {assertion} follows from \cite{ABHPS}.
\end{proof}

Finally we give an application to the enumeration problem of biCatalan combinatorics. Here we briefly recall the theory of Coxeter-biCatalan combinatorics. 
We refer to the original paper \cite{BR}, and also \cite{Re1,Re2,RS} for more {background details}.

Let $W$ be the Weyl group of type $\bbA_n$, which is isomorphic to the symmetric group of rank $n+1$.  
We regard $W$ as a poset defined by the (right) weak order $\leq$. 
Note that $W$ is a lattice, that is, for any $x,y\in W$, there is a greatest lower bound  $x\wedge y$ and 
a least upper bound $x\vee y$.

An equivalence relation $\equiv$ on $W$ is called a \emph{congruence} if it has the following property: 
If $x_1\equiv y_1$ and $x_2\equiv y_2$, then we have $x_1 \wedge x_2\equiv y_1 \wedge y_2$ and $x_1\vee x_2\equiv y_1\vee y_2$ for all $x_1,x_2,y_1,y_2\in W$.
Given a congruence $\Theta$ on $W$, 
we define the \emph{quotient lattice} $W/\Theta$, which is defined as follows: A $\Theta$-class $C_1$ is less than or equal to a $\Theta$-class $C_2$ in $W/\Theta$ if there exists an element $x_1$ of $C_1$ and an element $x_2$ of $C_2$ such that $x_1\le x_2$ in $W$.
For two congruences $\Theta$ and  $\Theta'$,
we define $\Theta \le \Theta'$ if, for any $x, y\in W$, $x \equiv_\Theta y$ implies $x \equiv_{\Theta'} y$. Then the set of congruences {turns out to be also a} lattice.

{Let $Q$ be a quiver of type $\bbA_n$. We define a Coxeter element $c$ of $W$ as an expression $s_{i_1}s_{i_2}\cdots s_{i_n}$ such that
if there is an arrow $j\to i$ in $Q$, then $s_i$ appears before $s_j$ in the expression $s_{i_1}s_{i_2}\cdots s_{i_n}$. It is uniquely determined as an element of $W$. 
We call $c$ a \emph{bipartite} Coxeter element if $Q$ has only sinks and sources. 
Moreover, we define a congruence $\Theta_c$ on $W$ called the \emph{$c$-Cambrian congruence}, which is defined as follows. 
Let $\textnormal{Hasse}W$ be the Hasse quiver of $W$ and 
$\varepsilon_c:=\{s_js_i\to s_j\ |\ j\to i\in Q_1\}$ of arrows of $\textnormal{Hasse}W$.   
We say that $\Theta$ contracts an arrow $x\to y$ in $\textnormal{Hasse}W$ if $x\equiv_{\Theta} y$. 
Then we define $\Theta_c$ to be the minimum congruence that contracts all arrows in $\varepsilon_c$} (we refer to \cite{Re1} for the precise definition). 
Then we define the quotient lattice $W/\Theta_c$ as above, which is called the \emph{Cambrian lattice} of $W$. 

{Now consider the usual representation of $W$
as a reflection group acting on the space $\bbR^{n}$ with trivial fixed subspace. The collection of hyperplanes that define the reflections 
is the \emph{Coxeter arrangement} of $W$. The
hyperplanes in the Coxeter arrangement cut the space into cones, which constitute a fan called the \emph{Coxeter fan} (see \cite{BB} for details). }
Moreover, define the \emph{Cambrian fan} $\Camb(W,c)$ by coarsening the Coxeter fan obtained by gluing together maximal cones according to an equivalence relation on $\Theta_c$. 
Then the maximal cones of $\Camb(W,c)$ correspond to the classes of $W/\Theta_c$ and they are naturally indexed by $c$-sortable elements of $W$ \cite{RS}.

Moreover, we take $\Theta_{biC}:=\Theta_c\wedge\Theta_{c^{-1}}$, that is, 
the greatest lower bound of $\Theta_c$ and $\Theta_{c^{-1}}$. 
Then, as in the same way as above, we define the quotient lattice 
$W/\Theta_{biC}$ (resp.\ the fan $\biCamb(W,c)$), called the \emph{biCambrian lattice} (resp.\ the \emph{biCambrian fan}) \cite{BR}. 
The biCambrian fan $\biCamb(W,c)$ is also defined as 
the coarsest common refinement of the two Cambiran fans $\Camb(W,c)$ and $\Camb(W,c^{-1})$.
Note that the maximal cones of $\biCamb(W,c)$ correspond to classes of $W/\Theta_{biC}$ and they are naturally indexed by  $c$-bisortable elements of $W$ \cite{BR}.

{We define the simplicial sphere underlying $\biCamb(W,c)$ to be the intersection of $\biCamb(W,c)$ and a unit sphere centered at the origin. 
In \cite{BR} (also in \cite{DIRRT}), it is shown that it is a simplicial complex and 
the $h$-vector has been studied \cite[Theorem 2.13]{BR}.} 
Then our result also implies the same consequence. 

\begin{thm}\label{bicatalan}
Let $W$ be the Weyl group of type $\bbA_n$ and $c$ a bipartite Coxeter element of $W$. 
The $f$-polynomial and the $h$-polynomial of 
the simplicial sphere underlying $\biCamb(W,c)$ are 
given as follows : 

$$f(x) = \sum_{j=0}^{n}\binom{n+j}{j,j,n-j}x^{n-j},\ \ \ h(x) = \sum_{j=0}^{n}\binom{n}{j}^2x^{n-j}.$$
\end{thm}

\begin{proof}
Let $G$ be a linear tree given by
$$G=(\xymatrix@M=0pt{\bullet\ar@{-}[r]^1& \ar@{-}[r]^2\bullet &\cdots\ar@{-}[r]& \ar@{-}[r]^n\bullet& \bullet}).
$$
Then the Brauer tree algebra $A_G$ has the following structure:
$$\psmat{
 1&\\&2\\1&}
\ds
\psmat{
 &2&\\1&&3\\&2&}
\ds\cdots\ds
\psmat{
 &n-1&\\n-2&&n\\&n-1&}
\ds \psmat{
 &n\\n-1&\\&n}.
$$
Recall that we call a complex $P$ in $\Kb(\proj A)$ \emph{silting} if $\Hom_{\Kb(\proj A)}(P,P[i])=0$ for any $i>0$ and {if} it satisfies (3) of Definition \ref{two-term tilt}. 
Then, according to \cite[Theorem 3.3]{Ad}, we have a poset isomorphism 
$$\ttilt A_G\overset{\simeq}{\to}\tsilt (A_G/\soc(A_G)),$$
where $\tsilt (A_G/\soc(A_G))$ is the set of isoclasses of basic
2-term silting complexes 
and $\soc(A_G)$ is the right socle of $A_G$, which coincides with the left socle and is a two-sided ideal of $A_G$.

Let $\Pi$ be the preprojective algebra of type $\bbA_{n}$, that is, the algebra $$\Bbbk(\xymatrix@C30pt@R10pt{
1 \ar[r]^{a_1} &2 \ar[r]^{a_2}\ar@<1ex>[l]^{a_1^*}&\ar@<1ex>[l]^{a_2^*}\cdots\ar[r]^{a_{n-2}}&n-1 \ar@<1ex>[l]^{a_{n-2}^*}\ar[r]^{a_{n-1}}& \ar@<1ex>[l]^{a_{n-1}^*} n})/I,$$
where $I=\langle\sum_{i=1}^{n} (a_i^*a_i-a_{i-1}a_{i-1}^* )\rangle$. 
Then, since $A_G/\soc(A_G)$ is isomorphic to $\Pi/\rad^2(\Pi)$, we have the following poset isomorphisms
by the above isomorphism and \cite[Theorem 7.10]{DIRRT}:
$$\ttilt A_G\simeq \tsilt (A_G/\soc(A_G))=\tsilt (\Pi/\rad^2(\Pi))\simeq \mathrm{Tor}(\Pi/\rad^2(\Pi))
\simeq W/\Theta_{biC},$$
where $\mathrm{Tor}(\Pi/\rad^2(\Pi))$ denotes the set of torsion classes of the category of finite dimensional $\Pi/\rad^2(\Pi)$-modules. 
Note that the partial order of 2-term silting complexes coincides with that 
of torsion classes \cite[Corollary 3.9]{AIR}, and hence we can apply this result.
Since the poset structure of $\ttilt A_G$ is entirely determined by the cones of $g$-vectors \cite[Corollary 6.13]{DIJ}, 
$\Delta(A_G)$ is combinatorially equivalent to the simplicial sphere underlying $\biCamb(W,c).$
Then the {assertion} 
follows by Theorem \ref{main3}.
\end{proof}

\section*{Acknowledgement}
The authors thank Osamu Iyama for informing us of a proof of Proposition \ref{prp:2-conv-whole-conv} and Lutz Hille for answering our questions about his paper. 
Y.M. thank Takahide Adachi, Ryoichi Kase and Toshitaka Aoki for their kind advice and useful discussions. 
Y.M. is grateful to Hugh Thomas for letting us know the unimodular triangulations. 
The authors gratefully thank the referee for the very helpful comments, which led to a great improvement of the paper.
{}

\end{document}